\documentclass{amsart}

\DeclareFontFamily{U}{mathb}{\hyphenchar\font45}
\DeclareFontShape{U}{mathb}{m}{n}{
      <5> <6> <7> <8> <9> <10> gen * mathb
      <10.95> mathb10 <12> <14.4> <17.28> <20.74> <24.88> mathb12
      }{}
\DeclareSymbolFont{mathb}{U}{mathb}{m}{n}
\DeclareMathSymbol{\righttoleftarrow}{3}{mathb}{"FD}

\usepackage[english]{babel}
\usepackage[utf8]{inputenc}
\usepackage{graphicx}
\usepackage{xcolor}
\usepackage{longtable}
\usepackage{float}

\usepackage[all]{xy}

\usepackage{amsmath}
\usepackage{amsthm}
\usepackage{amssymb}
\usepackage{tikz}
\usepackage{csquotes}
\usepackage{imakeidx} 
\usepackage{appendix}
\usepackage{tikz-cd}
\usepackage{verbatim}
\usepackage{enumitem}
\usepackage{multicol}
\usepackage{aliascnt}
\usepackage[backref=page]{hyperref}
\usepackage{cleveref}
\usepackage{mathtools}

\hypersetup{colorlinks=true,linkcolor=blue,anchorcolor=blue,citecolor=blue}

\newtheorem{thm}{Theorem}[section]
\crefname{thm}{Theorem}{Theorems} 

\newaliascnt{prop}{thm}       
\newtheorem{prop}[prop]{Proposition} 
\aliascntresetthe{prop}       
\crefname{prop}{Proposition}{Propositions} 

\newaliascnt{lemma}{thm}
\newtheorem{lemma}[lemma]{Lemma}
\aliascntresetthe{lemma}
\crefname{lemma}{Lemma}{Lemmas}

\newaliascnt{cor}{thm}

\aliascntresetthe{cor}
\crefname{cor}{Corollary}{Corollaries}

\newaliascnt{conj}{thm}

\aliascntresetthe{conj}
\crefname{conj}{Conjecture}{Conjectures}

\newaliascnt{question}{thm}

\aliascntresetthe{question}
\crefname{question}{Question}{Questions}

\theoremstyle{definition}
\newaliascnt{defin}{thm}

\aliascntresetthe{defin}
\crefname{defin}{Definition}{Definitions}

\theoremstyle{remark}
\newaliascnt{rmk}{thm}
\newtheorem{rmk}[rmk]{Remark}
\aliascntresetthe{rmk}
\crefname{rmk}{Remark}{Remarks}
\Crefname{rmk}{Remark}{Remarks}

\numberwithin{equation}{section}

\newcommand{\Q}{\mathbb Q}
\newcommand{\F}{\mathbb F}

\newcommand{\Z}{\mathbb Z}

\renewcommand{\P}{\mathbb P}
\newcommand{\Spec}{\operatorname{Spec}}

\newcommand{\A}{\mathbb A}
\newcommand{\mc}[1]{\mathcal{#1}}

\renewcommand{\phi}{\varphi}

\newcommand{\on}[1]{\operatorname{#1}}
\newcommand{\ang}[1]{\langle{#1}\rangle}

\newcommand{\bG}{\mathbb G}
\newcommand{\bZ}{\mathbb Z}

\newcommand{\mg}{\mathfrak g}

\newcommand{\rd}{\mathrm d}

\newcommand{\rE}{\mathrm E}
\newcommand{\rH}{\mathrm H}

\newcommand{\rR}{\mathrm R}

\newcommand{\Am}{\mathrm{Am}}

\newcommand{\Pic}{\mathrm{Pic}}
\newcommand{\End}{\mathrm{End}}
\newcommand{\Hom}{\mathrm{Hom}}

\newcommand{\eqto}{\stackrel{\lower1.5pt\hbox{$\scriptstyle\sim\,$}}\to}
\newcommand{\eqdashto}{\stackrel{\lower1.5pt\hbox{$\scriptstyle\sim\,$}}\dashrightarrow}

\newcommand{\actsfromright}{\righttoleftarrow}

\title{Birational invariance of higher Amitsur groups}

\address{CNRS\\
	Institut Galil\'ee\\
	Universit\'e Sorbonne Paris Nord\\    
	93430, Villetaneuse, France}

\author[F. Scavia]{Federico Scavia}
\email{scavia@math.univ-paris13.fr}

\author[Yu. Tschinkel]{Yuri Tschinkel}
\address{
  Courant Institute,
  251 Mercer Street,
  New York, NY 10012, USA}
\address{Simons Foundation\\
160 Fifth Avenue\\
New York, NY 10010\\
USA}
\email{tschinkel@cims.nyu.edu}

\author[Zh. Zhang]{Zhijia Zhang}

\address{
Courant Institute,
  251 Mercer Street,
  New York, NY 10012, USA
}

\email{zz1753@nyu.edu}

\date{\today}

\makeatletter
\@namedef{subjclassname@2020}{\textup{2020} Mathematics Subject Classification}
\makeatother


\definecolor{MyDarkGreen}{rgb}{0.0,0.5,0.0}

\begin{document}

	\begin{abstract}
	Let $k$ be a field of characteristic zero and $G$ a finite group. We prove that for all $n\geq 2$, the $n$th Amitsur group is a stable $G$-birational invariant of smooth projective $G$-varieties over $k$. This was previously known for $n=2,3$. For smooth projective $G$-varieties with free and finitely generated Picard group, we also prove that the vanishing of the $G$-equivariant universal torsor obstruction implies the vanishing of the $n$th Amitsur group, for all $n\geq 2$. This was known for $n=2$. Our approach allows for effective computations of these obstructions; we illustrate this with several examples. 
	\end{abstract}
	
	\maketitle

    \section{Introduction}
	\label{sect:intro}

Let $k$ be a field and $G$ a finite group. By definition, a $k$-variety is a separated $k$-scheme of finite type, and a $G$-variety over $k$ is a $k$-variety endowed with a (regular) $G$-action over $k$. Let $X$ and $Y$ be integral $G$-varieties over $k$. We say that $X$ and $Y$ are \emph{$G$-birationally equivalent} if there exists a $G$-equivariant birational map $X\dashrightarrow Y$, and \emph{stably $G$-birationally equivalent} if there exist integers $m,n\geq 0$ such that $X\times_k \P^m_k$ and $Y\times_k\P^n_k$ are $G$-birationally equivalent (here $G$ acts trivially on $\P^m_k$ and $\P^n_k$). A $G$-variety $X$ is said to be $G$-(stably) linearizable if it is (stably) $G$-birationally equivalent to $\P(V)$, and \emph{$G$-unirational} if there exists a dominant $G$-equivariant  rational map $\P(V)\dashrightarrow X$, where $V$ is a finite-dimensional linear $G$-representation over $k$. 

One of the main goals of $G$-equivariant algebraic geometry is to classify $G$-varieties up to $G$-birational or stable $G$-birational equivalence. A particularly important case is that of $k$-rational $G$-varieties of some fixed dimension $d\geq 1$ on which $G$ acts generically freely. In this setting, the classification problem is equivalent to the classification of conjugacy classes of finite subgroups of the Cremona group $\mathrm{Cr}_d(k)$. This is a major open problem for $d \ge 3$.

In order to prove that $X$ and $Y$ are not (stably) $G$-birationally equivalent, one typically has to find a (stable) $G$-birational invariant which distinguishes them. In this paper, we introduce the \emph{higher Amitsur groups}, and prove their stable $G$-birational invariance. This is an infinite sequence of abelian groups defined as follows: Suppose that $k[X]^\times=k^\times$ (for example, $X$ is smooth, projective and geometrically connected) and consider the Grothendieck spectral sequence for the composition of the functor of \'etale global sections and the functor of $G$-invariants
\begin{equation} \label{eqn:spect}
\mathrm{E}_2^{p,q}\coloneqq \rH^p(G,\rH^q(X,\mathbb{G}_m))\Longrightarrow \rH^{p+q}_G(X,\mathbb{G}_m).
\end{equation}
For all $p\geq 0$ we have differentials
\[\mathrm{d}_2^{p,1}\colon \rH^p(G,\on{Pic}(X))\longrightarrow \rH^{p+2}(G,k^\times).\]
For all $n\geq 2$, we define the $n$th Amitsur group of the $G$-variety $X$ as 
\begin{equation}\label{eq:amitsur-defin-intro}\Am^n(X\actsfromright G,\bG_m)\coloneqq\mathrm{Im}(\rH^{n-2}(G,\on{Pic}(X))\xlongrightarrow{\mathrm{d}_2^{n-2,1}} \rH^n(G,k^\times)).\end{equation}
The second Amitsur group, introduced in \cite[Section 6]{blancfinite}, admits a concrete description in terms of $G$-linearizability of line bundles; see \Cref{rmk:second-amitsur} below. 

Prior to the present work, the Amitsur groups had been defined only when $n\leq 4$ and $k$ is algebraically closed, via the Leray spectral sequence for the morphism of stacks $[X/G]\to \mathrm{B}_kG$ (see, for example, \cite[Section 5]{KT-burnsurv}); by \Cref{rmk:old-def} below, our definition specializes to the previous one when $k$ is algebraically closed. Before our work, no results were known about these groups for $n>4$; the proof of stable birational invariance in the case of $n=4$ in \cite[Section 5]{KT-burnsurv} was based on blowup formulas in Voevodsky's motivic cohomology, which discouraged the study of the differentials $\rd_2^{n-2,1}$ 
for higher $n$. 

More precisely, suppose that $k$ is algebraically closed and of characteristic zero, and that $X$ and $Y$ are smooth, projective, and irreducible. Assume further that $X$ and $Y$ are $G$-birationally equivalent. Then the equality
$$
\Am^n(X\actsfromright G,\bG_m)=\Am^n(Y\actsfromright G,\bG_m)
$$ 
was only known in the following cases:
\begin{itemize}
    \item when $n=2$, by \cite[Theorem 6.1]{blancfinite},
    \item when $n=3$, by \cite{KT-dp} and \cite{KT-uni}, and
    \item when $n=4$ and $X$ and $Y$ are $k$-rational, by \cite[Theorem 5.7]{KT-burnsurv}.
\end{itemize}
Moreover, if $X$ is $G$-unirational then $$
\Am^n(X\actsfromright G,\bG_m)=0\quad \text{ for  } \quad n=2, 3,4;
$$
see \cite[Section 5]{KT-burnsurv}. 
It is natural to wonder whether, for every $n\geq 2$, the Amitsur group $\Am^n(X\actsfromright G,\bG_m)$ is a stable $G$-birational invariant which vanishes when $X$ is $G$-unirational. \emph{We answer this question in the affirmative when $k$ is an arbitrary field of characteristic zero.}

More generally, let $S$ be a $k$-split $G$-torus, that is, $S$ is a group scheme over $k$ isomorphic to $\mathbb{G}_m^r$, for some integer $r\geq 1$, endowed with a $G$-action  via group scheme automorphisms. We have the spectral sequence \begin{equation} \label{eqn:spect2}
\mathrm{E}_2^{p,q}\coloneqq \rH^p(G,\rH^q(X,S))\Longrightarrow \rH^{p+q}_G(X,S),
\end{equation}
and hence, for all $p\geq 0$, the differentials
\begin{equation}\label{eq:differential-intro-s}\mathrm{d}_2^{p,1}\colon \rH^p(G,\rH^1(X,S))\longrightarrow \rH^{p+2}(G,S(k)).
\end{equation}
For all $n\geq 2$, we define 
\begin{equation}\label{eq:amitsur-defin-intro-s}\Am^n(X\actsfromright G,S)\coloneqq \mathrm{Im}(\rH^{n-2}(G,\rH^1(X,S))\xlongrightarrow{\mathrm{d}_2^{n-2,1}} \rH^n(G,S(k))).
\end{equation}
When $S=\mathbb{G}_m$ with trivial $G$-action, this recovers \eqref{eq:amitsur-defin-intro}. We prove that the group $\Am^n(X\actsfromright G,S)$ is a stably $G$-birational invariant of $X$, for all $n\geq 2$.

\begin{thm}\label{thm:intro-invariant}
    Let $k$ be a field of characteristic zero, $G$ a finite group, and $S$ a $k$-split $G$-torus. Let $X$ and $Y$ be stably $G$-birationally equivalent smooth projective geometrically irreducible $G$-varieties over $k$. Then, for all $n\geq 2$, we have
    \[\Am^n(X\actsfromright G,S)=\Am^n(Y\actsfromright G,S).\]
\end{thm}

Our methods also clarify the relation between higher Amitsur groups and the \emph{$G$-equivariant universal torsor obstruction}. Assume that the Picard group $\on{Pic}(X)$ of the $G$-variety $X$ is $\Z$-free and finitely generated, and let $T_{\mathrm{NS}}$ be the N\'eron-Severi torus of $X$, that is, a $k$-split $G$-torus whose character lattice is isomorphic to $\on{Pic}(X)$ as a $G$-module. Then 
$\rH^1(X,T_{\mathrm{NS}})=\on{End}(\on{Pic}(X))$
and the spectral sequence \eqref{eqn:spect2} for $S=T_{\mathrm{NS}}$ yields a homomorphism
$$
\mathrm{d}_2^{0,1}\colon\End(\Pic(X))^G \longrightarrow \rH^2(G,T_{\mathrm{NS}}(k)). 
$$
The non-vanishing of the class 
$$
\beta(X\actsfromright G)\coloneqq\mathrm{d}_2^{0,1}(\mathrm{Id}_{\on{Pic}(X)})\in \rH^2(G,T_{\mathrm{NS}}(k))
$$
is the obstruction to lifting the $G$-action from $X$ to a universal torsor $\mathcal T_X\to X$, and also an obstruction to unirationality of the $G$-action; see \cite[\S 5]{KT-burnsurv}; the passage to universal torsors
allowed one to establish stable linearizability of nonlinearizable $G$-actions in many new cases \cite{HTtors}. 

If the class $\beta(X\actsfromright G)$ is zero, then 
$\Am^2(X\actsfromright G,\mathbb{G}_m)=0$, by
\cite[Proposition 5.3]{KT-burnsurv}. 
It is then natural to wonder whether the condition $\beta(X\actsfromright G)=0$ implies the vanishing of $\Am^n(X\actsfromright G,\mathbb{G}_m)$, for all $n\geq 2$. We prove that this is indeed the case.

\begin{thm}
\label{thm:main2}
Let $k$ be a field, $G$ a finite group, and $X$ a smooth projective geometrically connected $G$-variety such that $\on{Pic}(X)$ is $\Z$-free and finitely generated. If $\beta(X\actsfromright G)=0$, then $\Am^n(X\actsfromright G,S)=0$, for all $n\ge 2$ and all $G$-tori $S$.     
\end{thm}

As we explain below, our result is more precise and clarifies the relation between $\beta(X\actsfromright G)$ and the groups $\Am^n(X\actsfromright G,S)$. Indeed, we exhibit a $4$-term extension of $\on{Pic}(X)$ by $k^\times$ with extension class $-\beta(X\actsfromright G)$ and such that, for every $k$-split $G$-torus $S$, tensoring this sequence with the cocharacter lattice of $S$ and taking the double connecting homomorphisms yields the differentials \eqref{eq:amitsur-defin-intro-s}.

	\Cref{thm:main2} was known (in characteristic zero) when $X$ is a toric variety, the open orbit $T\subset X$ is $G$-stable, and the $G$-action on $T$ is by group automorphisms, in which case the vanishing of $\beta(X\actsfromright G)$ is equivalent to the $G$-unirationality of $X$; see \cite[Theorem 6.1]{KT-uni}. For the purposes of finding obstructions to $G$-unirationality, it would suffice to compute $\beta(X\actsfromright G)$. However, to test stable $G$-birational equivalence of $G$-varieties, one needs to rely on higher Amitsur groups, since $\beta(X\actsfromright G)$ itself is not a $G$-birational invariant, only its vanishing is. (The group $\rH^2(G,T_{\mathrm{NS}}(k))$ is not a $G$-birational invariant.)

\medskip 

The main ingredient for the proofs of Theorems \ref{thm:intro-invariant} and \ref{thm:main2} is the observation that the Amitsur groups admit the following equivalent definition. Let $X$ be a smooth projective integral $G$-variety, and consider the exact sequence of $G$-modules
\begin{equation}\label{eq:intro-elementary-obstruction-sequence}
	1\longrightarrow k^\times \longrightarrow k(X)^\times \xlongrightarrow{\on{div}} \mathrm{Div}(X) \longrightarrow \mathrm{Pic}(X) \longrightarrow 0.  
\end{equation} 
Let $S$ be a $k$-split $G$-torus. Tensoring \eqref{eq:intro-elementary-obstruction-sequence} with the cocharacter lattice $\mathfrak{X}_*(S)$ of $S$, we obtain an exact sequence
\begin{equation}\label{eq:intro-elementary-obstruction-sequence-s}
	\begin{tikzcd}[column sep=small]
		1\arrow[r] & S(k)\arrow[r]  & S(k(X)) \arrow[r,"\mathrm{div}"] & \mathrm{Div}(X)\otimes_\Z\mathfrak{X}_*(S) \arrow[r] & \mathrm{Pic}(X)\otimes_\Z\mathfrak{X}_*(S) \arrow[r] & 0.   
	\end{tikzcd}
\end{equation}
This yields, for every $n\geq 2$, a connecting homomorphism
\begin{equation}\label{eq:partial-intro}\partial^n\colon \rH^{n-2}(G,\on{Pic}(X)\otimes_\Z\mathfrak{X}_*(S))\longrightarrow \rH^n(G,S(k)),
\end{equation}
which is given by cup product with the class of the extension \eqref{eq:intro-elementary-obstruction-sequence}. We prove in \Cref{thm:delta=d} that 
\[
\Am^n(X\actsfromright G,S)=\on{Im}(\partial^n),\quad\text{for all $S$ and $n\geq 2$.}
\] This is a consequence of a result of Skorobogatov \cite[Proposition 1.1]{skoro}. Using the equivalent description of Amitsur groups provided by \Cref{thm:delta=d}, it is not difficult to prove their invariance under blowups in smooth closed $G$-stable centers and under products with projective spaces on which $G$ acts trivially; see \Cref{prop:invariance-blowup-product}. Since the characteristic of $k$ is equal to zero, the $G$-equivariant weak factorization theorem \cite[Theorem 0.3.1]{abramovich2002torification} then implies \Cref{thm:intro-invariant}. Finally, a well-known argument from homological algebra shows that $\beta(X\actsfromright G)$ is equal to the opposite of the class of \eqref{eq:intro-elementary-obstruction-sequence} in $\on{Ext}^2_G(\on{Pic}(X),k^\times)$; see \Cref{prop:connect}. Thus, if $\beta(X\actsfromright G)=0$, then for all $n\geq 2$ and all $k$-split $G$-tori $S$ we have $\partial^n=0$, which by \Cref{thm:delta=d} implies that $\Am^n(X\actsfromright G,S)=0$, proving \Cref{thm:main2}.

\medskip

\Cref{thm:delta=d} also makes the higher Amitsur groups much more amenable to computation. We illustrate this with several examples. More precisely, in \Cref{sect:exam}, we compute all higher Amitsur groups for cyclic actions on projective spaces (in particular, we show that they depend nontrivially on the arithmetic of the field $k$) and for the faithful action of the Klein $4$-group on the projective line. In \Cref{sect:exam2}, we construct a toric smooth projective $G$-variety $X$, where $G$ is the Klein $4$-group, such that $\beta(X\actsfromright G)\neq 0$ (in particular, $X$ is not $G$-unirational) while all higher Amitsur groups vanish. Finally, in \Cref{sect:exam3}, we prove nontriviality of $\beta(X\actsfromright G)$ for the action of the modular (maximal-cyclic) group $G$ of order $16$ on a certain del Pezzo surface $X$ of degree $2$: this implies that $X$ is not $G$-unirational, even though every abelian subgroup of $G$ has fixed points on $X$ and all higher Amitsur groups vanish. This was an open case in \cite{TZ-uni}, and such a $G$-action is unique on del Pezzo surfaces over algebraically closed fields of characteristic zero; see \Cref{rmk:why-this-dp-2}. 

In Section~\ref{sect:geom-field}, inspired by the descent formalism of Colliot-Th\'el\`ene and Sansuc \cite{CTSansucDuke}, we generalize our results to the arithmetic setting by replacing the $G$-action on $X$ with the Galois group action on $X\times_k\bar k$. In fact, the second Amitsur group was originally considered in this context in \cite{amitsur} to study birational equivalence of Severi-Brauer varieties. Arithmetic analogs of higher Amitsur groups have not been considered, to our knowledge: for example, the fundamental paper \cite{CTSansucDuke} is focused on the vanishing of the {\em elementary obstruction}, which controls the existence of universal torsors, which in turn are central to arithmetic applications over number fields, e.g., Hasse principle and weak approximation. This asymmetry is natural, because the higher \enquote{equivariant} Amitsur groups have wider applicability in equivariant geometry than the higher \enquote{arithmetic} Amitsur groups do in the study of rational points. Indeed, if a $k$-variety $X$ satisfies $X(k)\neq\emptyset$, then the elementary obstruction for $X$ vanishes, and hence all higher arithmetic Amitsur groups of $X$ vanish. On the other hand, the computation of higher equivariant Amitsur groups for $k$-unirational $G$-varieties $X$ is very useful to obstruct their $G$-unirationality: indeed, there exist $G$-unirational varieties without $G$-fixed points. We also point out another important discrepancy between the two contexts, due to the lack of an analog of Hilbert's theorem 90 in the $G$-equivariant setting; see \Cref{rmk:difference}.

	\section{The Amitsur groups}
    \label{sect:ami}

	Let $k$ be a field, $G$ a finite group, $S$ a $k$-split $G$-torus, $\mathfrak{X}_*(S)$ the cocharacter lattice of $S$, and $X$ a smooth $G$-variety. Let $k[X]$ be the ring of regular functions on $X$, $k[X]^\times$ the group of units in $k[X]$, $k(X)$ the product of the function fields of the irreducible components of $X$, $\on{Div}(X)$ the abelian group of Weil divisors on $X$, and $\on{Pic}(X)$ the Picard group of $X$. 
    
    We have a $G$-equivariant exact sequence
	\begin{equation}\label{elementary-obstruction-sequence}
	1\longrightarrow k[X]^\times \longrightarrow k(X)^\times \xlongrightarrow{\on{div}} \mathrm{Div}(X) \longrightarrow \mathrm{Pic}(X) \longrightarrow 0,  
	\end{equation}
    where $\on{div}$ denotes the divisor map. For every $k$-scheme $U$, we have a $G$-equivariant isomorphism
    \begin{equation}\label{eq:evaluation-isom}
    \mathfrak{X}_*(S)\otimes k[U]^\times\xlongrightarrow{\sim} S(U),\qquad \chi\otimes f\mapsto \chi\circ f.
    \end{equation} 
    Since $\mathfrak{X}_*(S)$ is free as a $\Z$-module, tensoring \eqref{elementary-obstruction-sequence} with $\mathfrak{X}_*(S)$ and using \eqref{eq:evaluation-isom} we obtain a $G$-equivariant exact sequence
    \begin{equation}\label{elementary-obstruction-sequence-torus}
	\begin{tikzcd}[column sep=small]
	1\arrow[r] & S(k[X]) \arrow[r]  & S(k(X)) \arrow[rr,"\mathrm{div}"] && \mathrm{Div}(X)\otimes \mathfrak{X}_*(S) \arrow[r] & \mathrm{Pic}(X)\otimes \mathfrak{X}_*(S) \arrow[r] & 0. 
	\end{tikzcd}
	\end{equation}
   Let 
    \begin{equation}\label{eq:alpha-x-s}
    \alpha(X\actsfromright G,S)\in \on{Ext}^2_G(\on{Pic}(X)\otimes \mathfrak{X}_*(S),S(k[X]))
    \end{equation}
    be the class of this extension. For all $n\geq 2$, we obtain connecting maps
    \begin{equation}\label{eq:partial-x-s}
        \partial^n\colon \rH^{n-2}(G,\on{Pic}(X)\otimes \mathfrak{X}_*(S))\longrightarrow \rH^n(G,S(k[X])),\qquad c\mapsto c\cup\alpha(X\actsfromright G,S).
    \end{equation}

    The spectral sequence for $G$-equivariant \'etale cohomology with coefficients in $S$ has the form \begin{equation} \label{hochschild-serre}
\mathrm{E}_2^{p,q}\coloneqq \rH^p(G,\rH^q(X,S))\Longrightarrow \rH^{p+q}_G(X,S).
\end{equation}
By definition, this is the Grothendieck spectral sequence for the composition of the functor of global sections on the category of \'etale abelian sheaves and the functor of $G$-invariants. In particular, for all $p\geq 0$, we have differentials
\begin{equation}\label{eq:differential-grothendieck-s}\mathrm{d}_2^{p,1}\colon \rH^p(G,\rH^1(X,S))\longrightarrow \rH^{p+2}(G,S(k[X])).
\end{equation}
Similarly, the spectral sequence for $G$-equivariant Zariski cohomology with $S$ coefficients has the form \begin{equation} \label{hochschild-serre-zariski}
\mathrm{E}_2^{p,q}\coloneqq \rH^p(G,\rH_{\mathrm{Zar}}^q(X,S))\Longrightarrow \rH^{p+q}_{G,\mathrm{Zar}}(X,S).
\end{equation}
This is the Grothendieck spectral sequence for the composition of the functor of global sections for the Zariski abelian sheaf $S$ and the functor of $G$-invariants. In particular, for all $p\geq 0$, we have differentials
\begin{equation}\label{eq:differential-grothendieck-s-zariski}(\mathrm{d}_{\mathrm{Zar}})_2^{p,1}\colon \rH^p(G,\rH^1_{\mathrm{Zar}}(X,S))\longrightarrow \rH^{p+2}(G,S(k[X])).
\end{equation}
We also have $G$-module isomorphisms 
\begin{equation}\label{eq:pic-char-h1}
    \on{Pic}(X)\otimes_\Z\mathfrak{X}_*(S)\xlongrightarrow{\sim} \rH^1_{\mathrm{Zar}}(X,S)\xlongrightarrow{\sim}\rH^1(X,S).
\end{equation}
Here, the first map is given by $L\otimes\chi\mapsto \chi_*(L)$. The second map is the canonical map induced by the change-of-site morphism $X_{\text{\'et}}\to X_{\mathrm{Zar}}$, and it is an isomorphism by Grothendieck's generalization of Hilbert's theorem 90.

\begin{thm}\label{thm:delta=d}
		Let $k$ be a field, $G$ a finite group, $S$ a $k$-split $G$-torus, and $X$ a smooth $G$-variety. For all $n\geq 2$, we have a commutative square
    \[
    \begin{tikzcd}
        \rH^{n-2}(G,\on{Pic}(X)\otimes_\Z\mathfrak{X}_*(S)) \arrow[dr,"\partial^n"] \arrow[r,"\sim"]  & \rH^{n-2}(G,\rH^1_{\on{Zar}}(X,S)) \arrow[r,"\sim"] \arrow[d,"(\mathrm{d}_{\on{Zar}})_2^{n-2,1}"] & \rH^{n-2}(G,\rH^1(X,S)) \arrow[dl,"\mathrm{d}_2^{n-2,1}"] \\
        & \rH^n(G,S(k[X]))
    \end{tikzcd}
    \]
    where the horizontal maps are induced by \eqref{eq:pic-char-h1}, and where the maps $\partial^n$, $(\mathrm{d}_{\on{Zar}})^{n-2,1}_2$ and $\mathrm{d}^{n-2,1}_2$ are the maps of \eqref{eq:partial-x-s}, \eqref{eq:differential-grothendieck-s-zariski} and \eqref{eq:differential-grothendieck-s}, respectively.
	\end{thm}

    For every integer $n\geq 2$, we may thus define the \emph{$n$th Amitsur group} of the smooth $G$-variety $X$ with coefficients in the $G$-torus $S$ as
    \begin{equation}\label{eq:amitsur-definition}
        \Am^n(X\actsfromright G,S)\coloneqq \on{Im}(\partial^n)=\on{Im}((\mathrm{d}_{\on{Zar}})^{n-2,1}_2)=\on{Im}(\mathrm{d}^{n-2,1}_2).
    \end{equation}
    We call $n$ the \emph{degree} of the Amitsur group $\Am^n(X\actsfromright G,S)$. When $k[X]^\times=k^\times$, the definition given in \eqref{eq:amitsur-definition} recovers \eqref{eq:amitsur-defin-intro-s}. 
    
    The statement of \Cref{thm:delta=d} depends on sign conventions adopted in homological algebra. In this text, we follow \cite[Appendix~A]{KT-uni}. With different sign conventions, the left triangle in the statement of \Cref{thm:delta=d} may commute only up to sign. In any case, the definition in \eqref{eq:amitsur-definition} is independent of these choices.
    
    The proof of \Cref{thm:delta=d} relies on the following lemma.

\begin{lemma}\label{lem:d2}
    Let $\mc{A}$ and $\mc{B}$ be abelian categories, and assume that $\mc{A}$ has enough injectives. Let $\Phi\colon\mc{A}\to \mc{B}$ be a left exact additive functor. Consider a $2$-term complex $A^\bullet=[A^0\xlongrightarrow{f} A^1]$ concentrated in degrees $0$ and $1$. In the hypercohomology spectral sequence
    \begin{equation}\label{appendix-eq:hypercohomology}\mathrm{E}_2^{p,q}=\mathrm{R}^p\Phi(\mathrm{H}^q(A^\bullet))\Longrightarrow \mathrm{R}^{p+q}\Phi(A^\bullet),\end{equation}
    the differentials
    \begin{equation}\label{appendix-eq:hypercohomology-d2}
    d_2^{p,1}\colon \mathrm{R}^p\Phi(\rH^1(A^\bullet))\to \mathrm{R}^{p+2}\Phi(\rH^0(A^\bullet))\end{equation}
    are given by the double connecting homomorphisms for the exact sequence
    \[0\longrightarrow \on{Ker}(f)\longrightarrow A^0\xlongrightarrow{f} A^1\longrightarrow \on{Coker}(f)\to 0.\]
\end{lemma}

\begin{proof}
    This is standard; see for example \cite[Lemma A.3]{KT-uni} or, with different sign conventions, the more general \cite[Proposition 1.1]{skoro}.
\end{proof}

    \begin{proof}[Proof of \Cref{thm:delta=d}]
    The change-of-site map $X_{\text{\'et}}\to X_{\on{Zar}}$ gives rise to a morphism of spectral sequences from \eqref{hochschild-serre-zariski} to \eqref{hochschild-serre}, and hence to the commutative triangle on the right. It remains to prove the commutativity of the left triangle.
    
    By definition, the spectral sequence \eqref{hochschild-serre-zariski} is the hypercohomology spectral sequence for the complex of Zariski derived global sections $\rR\Gamma_{\on{Zar}}(X,S)$, viewed as an object of the bounded-below derived category of $G$-modules. Consider the $G$-equivariant morphism $j\colon\Spec (k(X))\to X$ given by the inclusion of the generic points of $X$. Let $X^{(1)}$ be the $G$-set of codimension-one points of $X$, and for every $x\in X^{(1)}$, let $i_x\colon \on{Spec}(k(x))\to X$ be the inclusion map. Let $\on{Div}_X\coloneqq \oplus_{x\in X^{(1)}}(i_x)_*\Z$ be the abelian Zariski sheaf of Weil divisors on $X$. We have a short exact sequence of $G$-equivariant abelian Zariski sheaves on $X$:
        \begin{equation}\label{eq:flasque-resolution-gm}
		1\longrightarrow \mathbb{G}_m \longrightarrow j_*\mathbb{G}_{m,\,k(X)} \xrightarrow{\ \operatorname{div}\ } \on{Div}_X\longrightarrow 0.
		\end{equation}
        Tensoring \eqref{eq:flasque-resolution-gm} with $\mathfrak{X}_*(S)$ yields the following short exact sequence of $G$-equivariant abelian Zariski sheaves on $X$:
		\begin{equation}\label{eq:flasque-resolution-s}
		1\longrightarrow S \longrightarrow j_*(S_{k(X)}) \xrightarrow{\ \operatorname{div}\ } \on{Div}_X\otimes_\Z\mathfrak{X}_*(S) \longrightarrow 0.
		\end{equation}
		This is a flasque resolution of $S$ in the category of $G$-equivariant abelian Zariski sheaves. (Note that it is not a flasque resolution in the category of $G$-equivariant abelian \'etale sheaves.) It follows that we have a quasi-isomorphism
        \begin{equation}\label{eq:rgamma-quasi-isomorphic-div}\rR\Gamma_{\on{Zar}}(X,S)\simeq [S(k(X))\xlongrightarrow{\on{div}} \on{Div}(X)\otimes_\Z\mathfrak{X}_*(S)]\end{equation}
        of complexes of $G$-modules concentrated in degrees $0$ and $1$. In particular, by \Cref{lem:d2}, the differential $(\mathrm{d}_{\on{Zar}})^{n-2,1}_2$ is equal to the double connecting homomorphism for the exact sequence
        \[1 \longrightarrow S(k[X]) \longrightarrow \rR\Gamma_{\on{Zar}}(X,S)^0 \longrightarrow \rR\Gamma_{\on{Zar}}(X,S)^1 \longrightarrow \rH_{\on{Zar}}^1(X,S) \longrightarrow 0.\]
        By \eqref{eq:rgamma-quasi-isomorphic-div}, this sequence is Yoneda-equivalent to the cohomology sequence
        \[1 \longrightarrow S(k[X]) \longrightarrow S(k(X)) \xlongrightarrow{\mathrm{div}} \on{Div}(X)\otimes_\Z\mathfrak{X}_*(S) \longrightarrow \rH^1_{\on{Zar}}(X,S) \longrightarrow 0\]
        associated to \eqref{eq:flasque-resolution-s}, and hence $(\mathrm{d}_{\on{Zar}})^{n-2,1}_2$ is equal to the double connecting homomorphism for this sequence as well. Finally, we have a commutative diagram with exact rows
        \[
        \begin{tikzcd}[column sep = small]
        1 \arrow[r] & S(k[X]) \arrow[r]\arrow[d,equal]  & S(k(X)) \arrow[r,"\on{div}"]\arrow[d,equal] & \on{Div}(X)\otimes_\Z\mathfrak{X}_*(S) \arrow[d,equal] \arrow[r] & \on{Pic}(X)\otimes_\Z\mathfrak{X}_*(S) \arrow[r] \arrow[d,"\wr"]  & 0  \\
        1 \arrow[r] & S(k[X]) \arrow[r] & S(k(X)) \arrow[r,"\on{div}"] & \on{Div}(X)\otimes_\Z\mathfrak{X}_*(S) \arrow[r] & \rH^1_{\on{Zar}}(X,S) \arrow[r] & 0, 
        \end{tikzcd}
        \]
        where the top row is \eqref{elementary-obstruction-sequence-torus}, and where the vertical isomorphism on the right comes from \eqref{eq:pic-char-h1}. Therefore, under the isomorphism \[\rH^{n-2}(G,\on{Pic}(X)\otimes_\Z\mathfrak{X}_*(S))\xlongrightarrow{\sim} \rH^{n-2}(G,\rH^1_{\on{Zar}}(X,S))\] induced by \eqref{eq:pic-char-h1}, the map $(\mathrm{d}_{\on{Zar}})^{n-2,1}_2$ corresponds to the double connecting map $\partial^n$, as desired.
	\end{proof}

We conclude this section with several remarks.

\begin{rmk}\label{rmk:comments-on-definition}
    Proving \Cref{thm:delta=d} for an arbitrary $k$-split $G$-torus $S$ instead of just $\mathbb{G}_m$ is essential for the application to the $G$-equivariant universal torsor obstruction given in \Cref{thm:main2}.
\end{rmk}

\begin{rmk}\label{rmk:negative-amitsur}
    Let $k$ be a field, $G$ a finite group, $S$ a $k$-split $G$-torus, and $X$ a smooth $G$-variety over $k$. The sequence \eqref{elementary-obstruction-sequence-torus} induces, for all integers $n$, double connecting homomorphisms
    \[\partial^n\colon \hat{\rH}^{n-2}(G,\on{Pic}(X)\otimes_\Z\mathfrak{X}_*(S))\longrightarrow \hat{\rH}^n(G,S(k[X])),\]
    where $\hat{\rH}$ denotes Tate cohomology; see \cite[Chapter I, \S 9]{neukirch2008cohomology}. We define, for all integers $n$, the $n$th Amitsur group of the $G$-variety $X$ as
    \[\Am^n(X\actsfromright G,S)\coloneqq \on{Im}(\partial^n).\]
    By \Cref{thm:delta=d}, this definition agrees with the one given in \eqref{eq:amitsur-definition}, for $n\geq 2$. In this paper, we will only work with Amitsur groups of degree $n\geq 2$.
\end{rmk}

\begin{rmk}\label{rmk:long-exact-sequence}
    As we showed in the proof of \Cref{thm:delta=d} (cf. \eqref{eq:rgamma-quasi-isomorphic-div}), the $\rE_2$-page of the spectral sequence \eqref{hochschild-serre-zariski} is concentrated in the rows $q=0$ and $q=1$. We obtain a long exact sequence
    \[
\begin{aligned}
0 &\to \rH^1(G,S(k[X])) 
\to \rH^1_{G,\mathrm{Zar}}(X,S) 
\to \rH^0\!\big(G,\on{Pic}(X)\otimes \mathfrak{X}_*(S)\big)
\xrightarrow{\partial^2} \dots \\
&\cdots \xrightarrow{\partial^n} \rH^{n}(G,S(k[X])) 
\to \rH^{n}_{G,\mathrm{Zar}}(X,S) 
\to \rH^{n-1}\!\big(G,\on{Pic}(X)\otimes \mathfrak{X}_*(S)\big) 
\xrightarrow{\partial^{n+1}}\cdots,
\end{aligned}
\]
where the maps $\partial^n$ have been defined in \eqref{eq:partial-x-s}.
\end{rmk}

   \begin{rmk}\label{rmk:old-def}
       Let $k$ be a field, $G$ a finite group, $S$ a $k$-split $G$-torus, and $X$ a smooth $G$-variety. Write $[X/G]$ for the quotient stack of $X$ by the $G$-action, and 
       $\mathrm{B}_kG\coloneqq [\on{Spec}(k)/G]$ for the classifying stack of $G$. Consider the Leray spectral sequence for the morphism $\pi\colon [X/G]\to \mathrm{B}_kG$ and the \'etale sheaf determined by $S$:
\begin{equation} \label{eqn:spect-leray}
\mathrm{E}_2^{p,q}\coloneqq \rH^p(\mathrm{B}_kG,\mathrm{R}^q\pi_*(S))\Longrightarrow \rH^{p+q}([X/G],S).
\end{equation}
    By definition, this is the Grothendieck spectral sequence for the composition of $\pi_*$ and the global sections functor over $\mathrm{B}_kG$. We have differentials
    \begin{equation}\label{eq:differential-leray-s}\mathrm{d}_2^{p,1}\colon \rH^p(\mathrm{B}_kG,\mathrm{R}^1\pi_*(S))\longrightarrow \rH^{p+2}(\mathrm{B}_kG,\pi_*(S)).
\end{equation}
    When $k$ is algebraically closed, there is an equivalence between the category of \'etale abelian sheaves over $\mathrm{B}_kG$ and the category of $G$-modules, under which the functor of global sections corresponds to the functor of $G$-invariants. Thus, when $k$ is algebraically closed, the two Grothendieck spectral sequences \eqref{eqn:spect-leray} and \eqref{hochschild-serre} are isomorphic and the differentials $\mathrm{d}_2^{p,1}$ of \eqref{eq:differential-grothendieck-s} and \eqref{eq:differential-leray-s} coincide, so that in particular our definition of $\Am^n(X\actsfromright G,S)$ coincides with that of \cite{KT-burnsurv}.
   \end{rmk}

\begin{rmk}\label{rmk:second-amitsur}
     The second Amitsur group (for $S=\mathbb{G}_m$) admits a concrete description in terms of equivariant line bundles. Let $k$ be a field, $G$  a finite group, and $X$ a smooth irreducible $G$-variety. Let 
     $
     \on{Pic}_G(X)\coloneqq \rH^1_G(X,\mathbb{G}_m)
     $ 
     be the group of $G$-equivariant isomorphism classes of $G$-linearized line bundles on $X$. The exact sequence of low-degree terms associated to \eqref{eqn:spect} shows that $\Am^2(X\actsfromright G,\bG_m)$ is the cokernel of the natural map $\Pic_G(X)\to \Pic(X)^G$, and in particular that it vanishes if and only if every $G$-invariant element of $\on{Pic}(X)$ is represented by a $G$-linearized line bundle on $X$. 
\end{rmk}

\section{Proof of Theorem \ref{thm:intro-invariant}}

Our proof of \Cref{thm:intro-invariant} relies on the next three assertions. 

\begin{lemma}\label{lem:pullback-alpha}
    Let $k$ be a field, $G$ a finite group, $S$ a $k$-split $G$-torus, $\pi\colon \tilde{X}\to X$ a dominant $G$-equivariant morphism of smooth irreducible $G$-varieties such that $k[X]^\times=k[\tilde{X}]^\times=k^\times$, and \[\pi^*\colon \on{Ext}^2_G(\on{Pic}(\tilde{X})\otimes_\Z\mathfrak{X}_*(S),S(k))\longrightarrow \on{Ext}^2_G(\on{Pic}(X)\otimes_\Z\mathfrak{X}_*(S),S(k))\]
    the induced map. Then $\pi^*(\alpha(\tilde{X}\actsfromright G,S))=\alpha(X\actsfromright G,S)$.
\end{lemma}

\begin{proof}
We have a commutative diagram of $G$-modules with exact rows
\[
	\begin{tikzcd}
		1\arrow[r] & k^\times\arrow[r]  & k(\tilde{X})^\times \arrow[r] & \mathrm{Div}(\tilde{X}) \arrow[r] & \mathrm{Pic}(\tilde{X}) \arrow[r] & 0 \\
		1\arrow[r] & k^\times\arrow[r] \arrow[u,equal]  & k(X)^\times \arrow[r] \arrow[u,"\pi^*"'] & \mathrm{Div}(X) \arrow[r] \arrow[u,"\pi^*"'] & \mathrm{Pic}(X)\arrow[r] \arrow[u,"\pi^*"'] & 0.    
	\end{tikzcd}
	\]
    Here we have used the fact that $X$ and $\tilde{X}$ are smooth to identify Weil divisors with Cartier divisors, and the fact that $\pi$ is dominant and that the $k$-varieties $X$ and $\tilde{X}$ are integral to have a well-defined pullback at the level of Cartier divisors. Since  $\mathfrak{X}_*(S)$ is free as an abelian group, tensoring with $\mathfrak{X}_*(S)$ yields the following commutative diagram with exact rows
	\[
	\begin{tikzcd}[column sep = small]
		1\arrow[r] & S(k)\arrow[r]  & S(k(\tilde{X})) \arrow[r] & \mathrm{Div}(\tilde{X})\otimes_\Z\mathfrak{X}_*(S) \arrow[r] & \mathrm{Pic}(\tilde{X})\otimes_\Z\mathfrak{X}_*(S) \arrow[r] & 0 \\
		1\arrow[r] & S(k)\arrow[r] \arrow[u,equal]  & S(k(X)) \arrow[r] \arrow[u,"\pi^*"'] & \mathrm{Div}(X)\otimes_\Z\mathfrak{X}_*(S) \arrow[r] \arrow[u,"\pi^*"'] & \mathrm{Pic}(X)\otimes_\Z\mathfrak{X}_*(S) \arrow[r] \arrow[u,"\pi^*"'] & 0.    
	\end{tikzcd}
	\]
	By \cite[Chapter III, Proposition 5.1]{maclane1967homology}, this shows that 
    \[
    \pi^*(\alpha(\tilde{X}\actsfromright G,S))=\alpha(X\actsfromright G,S).
    \qedhere\]
\end{proof}

\begin{lemma}\label{lem:pullback-alpha-epsilon}
    Let $k$ be a field, $G$ a finite group, $S$ a $k$-split $G$-torus, $X$ a smooth irreducible $G$-variety, $\Sigma$ a $G$-stable set of irreducible prime divisors on $X$. Write $\epsilon\colon \Z[\Sigma]\to \on{Pic}(X)$ for the $G$-module homomorphism sending a divisor to its associated line bundle, and consider the pullback map 
    \[\epsilon^*\colon\on{Ext}^2_G(\on{Pic}(X)\otimes_\Z\mathfrak{X}_*(S),S(k[X]))\longrightarrow \on{Ext}^2_G(\Z[\Sigma]\otimes_\Z\mathfrak{X}_*(S),S(k[X])).\]
    Then $\epsilon^*(\alpha(X\actsfromright G,S))=0$.
\end{lemma}

\begin{proof}
    Consider the commutative diagram of $G$-modules with exact rows
    \[
	\begin{tikzcd}
		1\arrow[r] & k[X]^\times\arrow[r]  & k(X)^\times \arrow[r] & \mathrm{Div}({X}) \arrow[r] & \mathrm{Pic}(X) \arrow[r] & 0 \\
		1\arrow[r] & k[X]^\times\arrow[r] \arrow[u,equal]  & k(X)^\times \arrow[r] \arrow[u,equal] & M \arrow[r,"f"] \arrow[u] & \Z[\Sigma]\arrow[r] \arrow[u,"\epsilon"'] & 0,    
	\end{tikzcd}
	\]
    where $M$ is defined so that the square on the right is cartesian. Tensoring with $\mathfrak{X}_*(S)$
    yields a commutative diagram
	\begin{equation}\label{second-factor}
	\begin{tikzcd}[column sep = small]
		1\arrow[r] & S(k[{X}])\arrow[r]  & S(k({X})) \arrow[r] & \mathrm{Div}({X})\otimes_{\Z}\mathfrak{X}_*(S) \arrow[r] & \mathrm{Pic}({X})\otimes_{\Z}\mathfrak{X}_*(S) \arrow[r] & 0 \\
		1\arrow[r] & S(k[{X}])\arrow[r] \arrow[u,equal]  & S(k({X})) \arrow[r] \arrow[u,equal] & M\otimes_\Z\mathfrak{X}_*(S) \arrow[r,"f\otimes\on{id}"] \arrow[u] & \Z[\Sigma]\otimes_\Z\mathfrak{X}_*(S) \arrow[r] \arrow[u,"\epsilon\otimes\on{id}"'] & 0,    
	\end{tikzcd}
	\end{equation}
	where the square on the right is cartesian. The obvious inclusion $\Z[\Sigma]\to \on{Div}({X})$ is a $G$-equivariant lift of $\epsilon$, and hence induces a $G$-equivariant splitting of $f$. Letting $N\coloneqq \on{Ker}(f)$, the bottom row is the composite (Yoneda product) of the short exact sequences of $G$-modules
	\[\begin{tikzcd}
		1\arrow[r] & S(k[X])\arrow[r]  & S(k(X)) \arrow[r] & N\otimes_\Z\mathfrak{X}_*(S) \arrow[r] & 0,  
	\end{tikzcd}\]
	\[\begin{tikzcd}
		0\arrow[r] & N\otimes_\Z\mathfrak{X}_*(S) \arrow[r]  & M\otimes_\Z\mathfrak{X}_*(S) \arrow[r,"f\otimes\on{id}"] & \Z[\Sigma]\otimes_\Z\mathfrak{X}_*(S) \arrow[r] & 0.  
	\end{tikzcd}\]
    As the map $f$ is $G$-equivariantly split, the second sequence splits. Since composition of short exact sequences is additive \cite[Chapter III, Theorem 5.3]{maclane1967homology}, we deduce that the bottom row of \eqref{second-factor} splits. The vanishing of $\epsilon^*(\alpha(X\actsfromright G,S))$ now follows from \cite[Chapter III, Proposition 5.1]{maclane1967homology}.
\end{proof}

\begin{prop}\label{prop:invariance-blowup-product}
    Let $k$ be a field, $G$ a finite group, $S$ a $k$-split $G$-torus, and $X$ a smooth irreducible $G$-variety such that $k[X]^\times =k^\times$. Let $\pi\colon \tilde{X}\to X$ be
    \begin{enumerate}
        \item[(i)] the blowup along a smooth closed $G$-stable subvariety $Z\subsetneq X$, or
        \item[(ii)] the first projection from $\tilde{X}\coloneqq X\times_k \mathbb P(V\oplus \mathbf{1})$, where $V$ is a finite-dimensional $G$-representation over $k$ and $\mathbf{1}$ denotes the trivial one-dimensional $G$-re\-pre\-sen\-ta\-tion over $k$.
    \end{enumerate}
    For all $n\geq 2$, we have 
    \[\Am^n(\tilde{X}\actsfromright G,S)=\Am^n(X\actsfromright G,S).
    \]
    Moreover, $\alpha(X\actsfromright G,S)=0$ if and only if $\alpha(\tilde{X}\actsfromright G,S)=0$.
\end{prop}
	
	\begin{proof} 
	In case (i), let $E\subset \tilde{X}$ be the exceptional divisor, and let $\Sigma$ be the $G$-set of irreducible components of $E$. In case (ii), let $\Sigma=\{H\}$, where $H\subset \tilde{X}$ is the pullback of a $G$-stable hyperplane in $\P(V\oplus\mathbf{1})$. Let $\epsilon\colon \Z[\Sigma]\to \on{Pic}(\tilde{X})$ be the induced map. Observe that $k[\tilde{X}]^\times=k[X]^\times=k^\times$. The maps $\pi$ and $\epsilon$ induce a $G$-module isomorphism
    \[(\pi^*,\epsilon)\colon \on{Pic}(X)\oplus\Z[\Sigma]\xlongrightarrow{\sim}\on{Pic}(\tilde{X})\]
    and therefore isomorphisms of abelian groups
    \begin{equation}\label{eq:blowup-ext}
\begin{aligned}
(\pi^*,\epsilon^*)\colon\ &\mathrm{Ext}^2_G(\mathrm{Pic}(\tilde{X})\otimes_\Z\mathfrak{X}_*(S),S(k)) \\
&\xrightarrow{\sim}\ \mathrm{Ext}^2_G(\mathrm{Pic}(X)\otimes_\Z\mathfrak{X}_*(S),S(k)) \oplus
\mathrm{Ext}^2_G(\Z[\Sigma]\otimes_\Z\mathfrak{X}_*(S),S(k)).
\end{aligned}
\end{equation}
        In order to conclude, it is enough to prove that
    \begin{equation}\label{eq:pi-epsilon-alpha-x-tilde}(\pi^*,\epsilon^*)(\alpha(\tilde{X}\actsfromright G,S))=(\alpha(X\actsfromright G,S),0).
    \end{equation}
    We have $\pi^*(\alpha(\tilde{X}\actsfromright G,S))=\alpha(X\actsfromright G,S)$, by \Cref{lem:pullback-alpha}. Moreover, we have 
    $
    \epsilon^*(\alpha(\tilde{X}\actsfromright G,S))=0,
    $ 
    by applying \Cref{lem:pullback-alpha-epsilon} to the $G$-variety $\tilde X$ and the $G$-set $\Sigma$.  This proves \eqref{eq:pi-epsilon-alpha-x-tilde}, as desired.
	\end{proof}

\begin{proof}[Proof of \Cref{thm:intro-invariant}]
    Since $\on{char}(k)=0$, by the $G$-equivariant weak factorization theorem \cite[Theorem 0.3.1]{abramovich2002torification}, it suffices to show that
    \[\Am^n(X\actsfromright G,S) = \Am^n(Y\actsfromright G,S)\]
    when (i) $Y$ is the blowup of $X$ along a smooth closed $G$-stable center $Z\subsetneq X$, and when (ii) $Y=X\times_k\P^d_k$, for some integer $d\geq 0$, where $G$ acts trivially on $\P^d_k$. This was proved in \Cref{prop:invariance-blowup-product}.
\end{proof}

	\section{Proof of Theorem \ref{thm:main2}}

The goal of this section is to prove \Cref{thm:main2}. Let $k$ be a field, $G$ a finite group, and $X$ a smooth irreducible $k$-variety such that the abelian group $\on{Pic}(X)$ is free and finitely generated. Let $T_{\mathrm{NS}}$ be the N\'eron-Severi torus of $X$. By definition, this is a $k$-split $G$-torus whose character lattice $\mathfrak X^*(T_{\mathrm{NS}})$ is isomorphic to $\on{Pic}(X)$ as a $G$-module. It follows that $\mathfrak X_*(T_{\mathrm{NS}})$ is isomorphic to $\on{Pic}(X)^\vee$. We have a $G$-module isomorphism
    \[ \on{Pic}(X)\otimes \on{Pic}(X)^\vee\xlongrightarrow{\sim}\on{End}(\on{Pic}(X)),\qquad L\otimes\varphi\mapsto \big(M\mapsto \varphi(M)L\big).\] 
    By \Cref{thm:delta=d} for $n=2$ and $S=T_{\mathrm{NS}}$, we have a commutative triangle
    \[
    \begin{tikzcd}
        \on{End}(\on{Pic}(X))^G \arrow[d,"\partial^2"] \arrow[r,"\sim"]  & \rH^1(X,T_{\mathrm{NS}})^G \arrow[dl,"\mathrm{d}_2^{0,1}"] \\
        \rH^2(G,T_{\mathrm{NS}}(k[X]))
    \end{tikzcd}
    \]
We define the \emph{$G$-equivariant universal torsor obstruction} for the $G$-variety $X$ as 
$$
\beta(X\actsfromright G)\coloneqq\partial^2(\mathrm{Id}_{\on{Pic}(X)})\in \rH^2(G,T_{\mathrm{NS}}(k[X])).
$$
Since $\on{Pic}(X)$ is $\Z$-free, the higher Ext groups $\on{Ext}^q_\Z(\on{Pic}(X),k[X]^\times)$ vanish, for all $q\geq 1$. It follows that the spectral sequence
\[\rE^{p,q}_2\coloneqq \rH^p(G,\on{Ext}^q_\Z(\on{Pic}(X),k[X]^\times))\Longrightarrow \on{Ext}^{p+q}_G(\on{Pic}(X),k[X]^\times)\]
collapses. As $T_{\mathrm{NS}}(k[X])=\on{Hom}_\Z(\on{Pic}(X),k[X]^\times)$, we obtain an isomorphism
\begin{equation}\label{eq:h2-ext}
    \rH^2(G,T_{\mathrm{NS}}(k[X]))\xlongrightarrow{\sim} \on{Ext}^2_G(\on{Pic}(X),k[X]^\times).
\end{equation}

    \begin{prop}\label{prop:connect}
    Let $k$ be a field, $G$ a finite group, and $X$ an irreducible smooth projective $G$-variety such that $\on{Pic}(X)$ is free and finitely generated. Under the isomorphism \eqref{eq:h2-ext}, the class $\beta(X\actsfromright G)$ corresponds to $-\alpha(X\actsfromright G,\mathbb{G}_m)$, where $\mathbb{G}_m$ is viewed as a $k$-split $G$-torus with trivial action.
    \end{prop}

    \begin{proof}
    See \cite[Proposition 5.4]{KT-uni}. The result there is stated only when $X$ is $k$-rational, but this is not needed in the proof.
    \end{proof}

\begin{proof}[Proof of \Cref{thm:main2}]
    Suppose that $\beta(X\actsfromright G)=0$. By \Cref{prop:connect}, we have $\alpha(X\actsfromright G,\mathbb{G}_m)=0$, that is, the extension \eqref{elementary-obstruction-sequence} is trivial. Let $S$ be a $k$-split $G$-torus. Since \eqref{elementary-obstruction-sequence-torus} is obtained from \eqref{elementary-obstruction-sequence} by tensoring with $\mathfrak{X}_*(S)$, it is also trivial. It follows that, for all $n\geq 2$, the connecting homomorphism $\partial^n$ of \eqref{eq:partial-x-s} vanishes, and hence by \Cref{thm:delta=d} we conclude that $\Am^n(X\actsfromright G,S)=0$.
\end{proof}

\section{Summary of properties of higher Amitsur groups}
\label{sect:summary}

    \begin{thm}
\label{thm:summary}
Let $k$ be a field of characteristic zero, $S$ a $k$-split $G$-torus, $X$ an irreducible smooth projective $G$-variety over $k$. For all integers $n\geq 2$, the groups 
$$
\Am^n(X\actsfromright G,S)
$$ 
satisfy the following: 
\begin{enumerate}
    \item they are stable $G$-birational invariants of $X$,
    \item they vanish when $X^G\neq \varnothing$, 
    \item if $\Am^n(X\actsfromright G_p,S)=0$, for every prime $p$ and every $p$-Sylow subgroup $G_p$ of $G$, then $\Am^n(X\actsfromright G,S)=0$,
    \item given a $G$-equivariant rational map $\phi\colon Y\dashrightarrow X$ of smooth projective varieties with regular $G$-action  one has
$$
\Am^n(X\actsfromright G,S)\subseteq \Am^n(Y\actsfromright G,S),
$$
with equality if 
$\phi$ is a $G$-equivariant morphism inducing
an isomorphism $\Pic(X)\simeq \Pic(Y)$.  
\item 
for a linear representation $V$ of $G$ one has
$$
\Am^n(\P(V)\actsfromright G,S)=0, 
$$
\item if $X$ is $G$-unirational then 
$$
\Am^n(X\actsfromright G,S)=0.
$$
\end{enumerate} 
\end{thm}

\begin{proof}
Property (1) follows from \Cref{thm:intro-invariant}. 
Property (2) is an immediate consequence of functoriality of the spectral sequence \eqref{hochschild-serre}. Property (3) follows from a restriction-corestriction argument in group cohomology. For property (4), let
\[
\xymatrix{
  & \tilde{ Y}  \ar@{->}[dl] \ar@{->}[dr] &  \\
    Y \ar@{-->}^{\varphi}[rr] && X
}
\]
be a $G$-equivariant resolution of indeterminacies of $\phi$, where $\tilde{Y}$ is smooth. By (1), 
$$
\mathrm{Am}^n(Y\actsfromright G,S)= \mathrm{Am}^n(\tilde{Y}\actsfromright G,S).
$$
By functoriality of the spectral sequence \eqref{hochschild-serre}, 
$$
\mathrm{Am}^n(X\actsfromright G,S)\subseteq \mathrm{Am}^n(\tilde{Y}\actsfromright G,S)= 
\mathrm{Am}^n(Y\actsfromright G,S).
$$
Property (5)
follows from the natural projection $\P(V\oplus \mathbf 1)\dashrightarrow\P(V)$, combined with Properties (2) and (4).
Property (6) is consequence of (4) and (5). 
\end{proof}

\begin{rmk}
    A $G$-unirational smooth projective $G$-variety over a field $k$ does not necessarily have $G$-fixed points, even if $k$ is algebraically closed. Thus Property (6) is not a direct consequence of Property (2).
\end{rmk}

\section{Actions on projective spaces}
\label{sect:exam}

The following lemma gives a way to compute the class $\alpha(X\actsfromright G,S)$ of \eqref{eq:alpha-x-s}, and the groups $\mathrm{Am}^n(X\actsfromright G,S)$ in practice.

\begin{lemma}\label{lem:finitepresent}
    Let $k$ be a field, $S$ a $k$-split $G$-torus, $X$ an irreducible smooth $G$-variety, and $\{D_i\}_{i\in I}$ a finite $G$-invariant collection of irreducible divisors of $X$ whose classes generate $\Pic(X)$, and  $U\coloneqq X\setminus(\cup_{i\in I}D_i)$. Then the class of the $G$-equivariant exact sequence 
    \begin{align*}
    0\longrightarrow S(k[X]) \longrightarrow S(k[U])\longrightarrow (\oplus_{i\in I}\Z\cdot D_i)\otimes_\Z\mathfrak{X}_*(S)\longrightarrow\Pic(X)\otimes_\Z\mathfrak{X}_*(S)\longrightarrow 0
\end{align*}
in $\mathrm{Ext}_G^2(\Pic(X)\otimes_\Z\mathfrak{X}_*(S),S(k[X]))$ is equal to the class $\alpha(X\actsfromright G,S)$ of \eqref{eq:alpha-x-s}.
\end{lemma}
\begin{proof}
    We have a $G$-equivariant commutative diagram of with exact rows
    $$
  \begin{tikzcd}
		1\arrow[r] & k[X]^\times\arrow[r]  & k(X)^\times \arrow[r] & \mathrm{Div}({X}) \arrow[r] & \mathrm{Pic}({X}) \arrow[r] & 0 \\
		1\arrow[r] & k[X]^\times\arrow[r] \arrow[u,equal]  & k[U]^\times \arrow[r] \arrow[u,hook]  & \bigoplus_{i\in I}\Z\cdot D_i \arrow[r] \arrow[u,hook] &\Pic(X) \arrow[r] \arrow[u,equal] & 0,
	\end{tikzcd}
    $$
    where the vertical maps are the obvious inclusions. The conclusion follows from tensorization with $\mathfrak{X}_*(S)$, followed by an application of \cite[Chapter III, Proposition 5.1]{maclane1967homology}.
\end{proof}

The following example shows that the Amitsur groups depend on the ground field $k$, in general.

\begin{prop}
    Let $k$ be a field, $m \ge 2$ an integer, and $b \in k^\times$. Let $G = \langle \sigma \rangle$ be a cyclic group of order $m$. Consider the projective space $X = \mathbb{P}^{m-1}_k$ with homogeneous coordinates $x_0, \ldots, x_{m-1}$, equipped with a $G$-action defined by
\[ \sigma(x_i) = x_{i+1} \quad \text{for } 0 \le i \le m-2, \quad \text{and} \quad \sigma(x_{m-1}) = b x_0. \]
We have $\beta(X\actsfromright G)=0$ if and only if $b\in k^{\times m}$. Moreover, $\Am^n(X\actsfromright G,\mathbb{G}_m)=0$ if $n$ is odd, and $\Am^n(X\actsfromright G,\mathbb{G}_m)$ is isomorphic to the subgroup generated by $b$ in $k^\times/k^{\times m}$ if $n$ is even.
\end{prop}

\begin{proof}
    We identify the $G$-module $\on{Pic}(\P^{m-1}_k)$ with the trivial $G$-module $\Z$ via the degree isomorphism. For every $i=0,\dots,m-1$, let $D_i\subset \P^{m-1}_k$ be the hyperplane given by the equation $x_i = 0$. Let $U\coloneqq \P^{m-1}_k\setminus \left(\cup_{i=0}^{m-1} D_i\right)$. By \Cref{lem:finitepresent}, $\beta(X\actsfromright G)$ coincides with the class of the extension
\[1 \longrightarrow k^\times \longrightarrow k[U]^\times \xrightarrow{\operatorname{div}} \bigoplus_{i=0}^{m-1}\Z\cdot D_i \xrightarrow{\deg} \mathbb{Z} \longrightarrow 0.\]
Consider the following commutative diagram with exact rows:
\[
\begin{tikzcd}
    0 \arrow[r] & \Z \arrow[r] \arrow[d,"f_b"] & \Z[G] \arrow[r,"\sigma-1"] \arrow[d] & \Z[G] \arrow[d,"\wr"] \arrow[r] & \Z \arrow[d,equal] \arrow[r]  & 0 \\
    1 \arrow[r] & k^\times \arrow[r] & k[U]^\times \arrow[r,"\on{div}"] & \bigoplus_{i=0}^{m-1}\Z\cdot D_i \arrow[r,"\on{deg}"] & \mathbb{Z} \arrow[r] & 0.
\end{tikzcd}
\]
Here, $f_b$ sends $1\in \Z$ to $b$, the second vertical arrow sends $1\in \Z[G]$ to $x_0/x_1$, the third vertical arrow sends $1\in \Z[G]$ to $D_0$, and $\bZ\to\bZ[G]$ sends $1$ to $\sum_{i=1}^m\sigma^i$. The top row of this diagram is a generator $\chi\in \on{Ext}^2_G(\Z,\Z)$. By \cite[Chapter III, Proposition 5.1]{maclane1967homology}, we deduce that the map $(f_b)_*\colon \on{Ext}^2_G(\Z,\Z) \longrightarrow \on{Ext}^2_G(\Z,k^\times)$ sends $\chi$ to $\beta(X\actsfromright G)$.
Consider the commutative square
\[
\begin{tikzcd}
    \Z/m\Z \arrow[r,"\sim"] \arrow[d,"f_b"] & \rH^2(G,\Z) \arrow[r,"\sim"] \arrow[d,"(f_b)_*"] & \on{Ext}^2_{G}(\Z,\Z) \arrow[d,"(f_b)_*"] \\
    k^\times/k^{\times m} \arrow[r,"\sim"] & \rH^2(G,k^\times) \arrow[r,"\sim"] & \on{Ext}^2_{\Z}(\Z,k^\times),
\end{tikzcd}
\]
where the isomorphisms on the left are given by cup product with $\chi$; see \cite[Proposition 1.7.1]{neukirch2008cohomology}. Thus $\beta(X\actsfromright G)=0$ if and only if $b=1$ in $k^\times/k^{\times m}$, as desired. Moreover, it follows from \cite[Proposition 1.7.1]{neukirch2008cohomology} that for all even $n\geq 2$ the map $\partial^n$ of \eqref{eq:partial-x-s} sends $1\in \Z$ to the class of $b\cup \chi^{n/2}$ in $\rH^n(k,G)$. On the other hand, if $n$ is odd, then $\rH^{n-2}(G,\Z)=0$, and hence $\partial^n=0$. Thus $\Am^n(X\actsfromright G,\mathbb{G}_m)$ has the desired form for all $n\geq 2$.
\end{proof}

When $k$ is algebraically closed, there exists a unique (up to conjugation) faithful action of the Klein $4$-group on $\P^1_k$. We compute all higher Amitsur groups for this action.

\begin{prop}\label{prop:p^1-amitsur}
Let $G\coloneqq \mathbb Z/2\times \mathbb Z/2$ be the Klein $4$-group, and let
$\sigma,\tau\in G$ be generators. Let $k$ be a field of characteristic not equal to two, let $x_0,x_1$ be homogeneous coordinates on $\mathbb P^1_k$, and consider the $G$-action on $\mathbb P^1_k$ given by
\[
\sigma\cdot (x_0:x_1)=(x_1:x_0),
\qquad
\tau\cdot (x_0:x_1)=(-x_0:x_1).
\]
The map $\partial^n$ of \eqref{eq:partial-x-s} (for $S=\mathbb{G}_m$) is injective for every $n\geq 3$, and the map
\[
\partial^2\colon
\rH^0(G,\operatorname{Pic}(\mathbb P^1_k))
\longrightarrow
\rH^2(G,k^\times)
\]
has kernel $2\mathbb Z$ and image isomorphic to $\mathbb Z/2$. In particular, 
\[
\operatorname{Am}^n(\mathbb P^1_k\actsfromright G,\mathbb G_m)
\simeq 
\begin{cases}
(\mathbb Z/2)^{\frac n2}, & n\geq 2 \text{ even},\\[2mm]
(\mathbb Z/2)^{\frac{n-3}{2}}, & n\geq 3 \text{ odd}.
\end{cases}
\]
\end{prop}
The group $\operatorname{Am}^n(\mathbb P^1_k\actsfromright G,\mathbb{G}_m)$ was computed in \cite[Proposition 6.4]{blancfinite} for $n=2$ and in \cite[Example 5.9]{KT-burnsurv} for $n=4$.

\begin{proof}
We have $\rH^0(G,\Z)=\Z$ and, by the K\"unneth formula with integer coefficients, for all $n\geq 3$  the group $\rH^{n-2}(G,\Z)$ is isomorphic to $(\Z/2)^{\frac{n}{2}}$ when $n$ is even and to $(\Z/2)^{\frac{n-3}{2}}$ when $n$ is odd. Therefore, it suffices to prove that $\partial^2$ has kernel $2\Z$ and that $\partial^n$ is injective for all $n\geq 3$. 

Let $\bar k$ be an algebraic closure of $k$. Extension of scalars gives, for all $n\geq 2$, a commutative square
\begin{equation}\label{eq:klein-k-kbar}
\begin{tikzcd}
\rH^{n-2}(G,\operatorname{Pic}(\mathbb P^1_k))
\arrow[r,"\sim"]
\arrow[d,"\partial^n"']
&
\rH^{n-2}(G,\operatorname{Pic}(\mathbb P^1_{\bar k}))
\arrow[d,"\partial^n"]
\\
\rH^n(G,k^\times)
\arrow[r]
&
\rH^n(G,\bar k^\times).
\end{tikzcd}
\end{equation}
Since $\rH^2(G,k^\times)$ and $\rH^2(G,\Bar{k}^\times)$ are $2$-torsion, the kernel of $\partial^2$ (over $k$ and $\Bar{k}$) contains $2\Z$. Therefore, in order to prove \Cref{prop:p^1-amitsur} for $n=2$, it remains to show that $\partial^2$ is not zero, and by \eqref{eq:klein-k-kbar} this can be checked over $\Bar{k}$. Moreover, if for all $n\geq 3$ the right vertical map in \eqref{eq:klein-k-kbar} is injective, then so is the left vertical map. Hence it suffices to prove \Cref{prop:p^1-amitsur} after replacing $k$ by $\bar k$. From now on, we assume that $k$ is algebraically closed. Set $t=x_0/x_1$, so that $\sigma(t)=t^{-1}$ and $\tau(t)=-t$. By \Cref{lem:finitepresent}, the class $\alpha(\mathbb P^1_k\actsfromright G,\mathbb G_m)$ is represented by the $G$-equivariant exact sequence
\begin{equation}\label{eq:klein-finite-present}
1\longrightarrow k^\times
\longrightarrow k^\times\cdot t^\Z 
\xlongrightarrow{\on{div}} \mathbb Z[0]\oplus \mathbb Z[\infty]
\longrightarrow \operatorname{Pic}(\mathbb P^1_k)
\longrightarrow 0.
\end{equation}
After identifying $\operatorname{Pic}(\mathbb P^1_k)$ with $\mathbb Z$ 
via the degree map, and letting $R\coloneqq \mathbb Z r$, where $r=[0]-[\infty]$, the above exact sequence is the Yoneda product of the two short exact sequences
\begin{equation}\label{eq:klein-first}
1\longrightarrow k^\times
\longrightarrow k^\times\cdot t^{\mathbb Z}
\longrightarrow R
\longrightarrow 0
\end{equation}
and
\begin{equation}\label{eq:klein-second}
0\longrightarrow R
\longrightarrow \mathbb Z[0]\oplus \mathbb Z[\infty]
\xlongrightarrow{\deg}
\mathbb Z
\longrightarrow 0.
\end{equation}
Let $x,y\in \rH^1(G,\F_2)$ be the dual basis of $\sigma,\tau$. For every $c\in \rH^1(G,\F_2)$, we define \(M_c\coloneqq \mathbb F_2\oplus \mathbb Z\), with \(G\)-action
\[
g\cdot(\epsilon,n)=(\epsilon+n c(g),n).
\]
Then $c$ is represented by the $G$-equivariant $\Z$-split short exact sequence
\[0 \longrightarrow \F_2 \longrightarrow M_c \longrightarrow \Z \longrightarrow 0.\]
The commutative diagram with exact rows
\[
\begin{tikzcd}[column sep=large]
0 \arrow[r] &
R \arrow[r] \arrow[d] &
\mathbb Z[0]\oplus \mathbb Z[\infty] \arrow[r,"\deg"] \arrow[d,"\phi"] &
\mathbb Z \arrow[r] \arrow[d,equal] &
0
\\
0 \arrow[r] &
\mathbb F_2 \arrow[r] &
M_{x} \arrow[r] &
\mathbb Z \arrow[r] &
0,
\end{tikzcd}
\]
where the left vertical map sends $r$ to $1\in \F_2$, and where $\phi$ is defined by 
$$
\phi([0])=(0,1), \quad \phi([\infty])=(1,1),
$$
shows that \(x\) is the image of \eqref{eq:klein-second} under the induced map $\rH^1(G,R)\to \rH^1(G,\mathbb F_2)$. On the other hand, the commutative diagram with exact rows
\[
\begin{tikzcd}[column sep=large]
0 \arrow[r] &
\F_2
\arrow[r]
\arrow[d,"\iota"] &
M_{y}\otimes R
\arrow[r]
\arrow[d] &
R
\arrow[r]
\arrow[d,equal] &
0
\\
1 \arrow[r] &
k^\times
\arrow[r] &
k^\times\cdot t^{\mathbb Z}
\arrow[r] &
R
\arrow[r] &
0
\end{tikzcd}
\]
where the vertical map in the middle is given by $$
(\epsilon,n)\otimes r\mapsto (-1)^{\epsilon }t^{n},
$$
shows that \eqref{eq:klein-first} is obtained from \(y\otimes R\) by the
pairing \(\F_2\otimes R\to \F_2\), followed by pushout along the inclusion 
\(\iota\colon \F_2\hookrightarrow k^\times\). The compatibility of the cup product with respect to maps of pairings now shows that the class of \eqref{eq:klein-finite-present} is the image of $y\cup x$ under the map $\iota_*\colon \rH^2(G,\F_2)\to \rH^2(G,k^\times)$, that is,
\[
\alpha(\mathbb P^1_k\actsfromright G,\mathbb G_m)
=
\iota_\ast(y\cup x)\quad \text{in $\rH^2(G,k^\times)$}.
\]
Let 
$$
\delta^n_K\colon \rH^n(G,k^\times)\longrightarrow \rH^{n+1}(G,\mu_2)
$$ 
be the connecting homomorphism associated to the Kummer exact sequence
\[
0\longrightarrow \F_2
\xlongrightarrow{\iota} k^\times
\xrightarrow{(-)^2}
k^\times
\longrightarrow 1.
\]
In order to show that $\partial^n$ is injective, it suffices to show that $\delta^n_K\circ\partial^n$ is injective. The diagram of short exact
sequences
\[
\begin{tikzcd}[column sep=large]
0 \arrow[r] &
\F_2 \arrow[r] \arrow[d,equal] &
\Z/4\Z \arrow[r] \arrow[d,hook] &
\F_2 \arrow[r] \arrow[d,hook,"\iota"] &
0
\\
0 \arrow[r] &
\F_2 \arrow[r,"\iota"] &
k^\times \arrow[r,"(-)^2"] &
k^\times \arrow[r] &
1,
\end{tikzcd}
\]
where the vertical map in the middle sends $1$ to some chosen $4$th root of unity in $k^\times$, shows that
\[
\delta^2_K(\alpha)
=
\delta^2_K(\iota_\ast(y\cup x))
=
\operatorname{Sq}^1(y\cup x)
=
y\cup x \cup (y+x)\quad \text{in $\rH^3(G,\F_2)$,}
\]
where in the last equality we have used that $\on{Sq}^1(x)=x^2$, $\on{Sq}^1(y)=y^2$ and the Cartan formula. We have a commutative square
\[
\begin{tikzcd}
    \rH^{n-2}(G,\Z) \arrow[r,"\delta^n_K\circ \partial^n"] \arrow[d,"\pi_2"] & \rH^{n+1}(G,\F_2) \\
    \rH^{n-2}(G,\F_2) \arrow[ur,hook,"y\cup x\cup (y+x)"']
\end{tikzcd}
\]
where $\pi_2$ is given by reduction modulo $2$. By the K\"unneth formula with $\F_2$ coefficients, $\rH^*(G,\F_2)\simeq \F_2[x,y]$, and hence in particular the map of multiplication by $y\cup x\cup (y+x)$ is injective. It follows that $\mathrm{Ker}(\delta_K^n\circ\partial^n)=\mathrm{Ker}(\pi_2)$. We have an exact sequence
\[\rH^{n-2}(G,\Z)\xlongrightarrow{\times 2}\rH^{n-2}(G,\Z)\xrightarrow{\pi_2}\rH^{n-2}(G,\F_2).\]
If $n=2$, this implies that $\mathrm{Ker}(\delta_K^n\circ\partial^n)=\mathrm{Ker}(\pi_2)=2\Z$. If $n\geq 3$, then $\rH^{n-2}(G,\Z)$ is $2$-torsion, and hence $\pi_2$ is injective. Thus $\delta^n_K\circ \partial^n$ is also injective.
\end{proof}

\section{A non-\texorpdfstring{$G$}{G}-unirational toric variety with trivial Amitsur groups}\label{sect:exam2}

For an algebraically closed field $k$ and a finite group $G$, a $G$-variety $X$ is said to be \emph{toric} if $X$ is a toric variety over $k$ on which $G$ acts (regularly) so that the torus orbit $T\subset X$ is $G$-stable. Note that we do not require $G$ to act on the torus $T$ by group automorphisms. In this section we prove the following theorem.

\begin{thm}\label{thm:connecting-maps-are-zero}
    Let $k$ be an algebraically closed field of characteristic zero and $G\coloneqq \Z/2\times \Z/2$. There exists a toric smooth projective $G$-variety $X$ such that $\beta(X\actsfromright G)\neq 0$, but $\Am^n(X\actsfromright G,\mathbb{G}_m)=0$, for all $n\geq 2$. In particular, $X$ is not $G$-unirational.
\end{thm}

We begin by recalling a result of \cite{KT-uni}.

\begin{prop}\label{prop:nonproper-toric-KT}
Let $k$ be an algebraically closed field of characteristic zero, $Y$ a smooth toric $k$-variety, $T\subset Y$ the torus orbit, and $G$ a finite group acting (regularly) on $Y$ such that $T$ is $G$-stable. Assume that $k[Y]^\times=k^\times$, that the finitely generated abelian group $\Pic(Y)$ is free, and that $\beta(Y\actsfromright G)=0$. Then $Y$ is $G$-unirational.
\end{prop}

\begin{proof}
    In view of \Cref{prop:connect}, this is \cite[Theorem 6.1]{KT-uni} when $Y$ is projective. One can check that, in the proof of \cite[Theorem 6.1]{KT-uni}, the assumption that $Y$ is projective is used only to ensure that $k[Y]^\times=k^\times$ and that $\on{Pic}(Y)$ is free.
\end{proof}

\begin{lemma}\label{lem:all-connecting-maps-are-zero}
    Let $G\coloneqq \Z/2\times \Z/2$ be the Klein $4$-group and $k$ an algebraically closed field of characteristic zero. There exist a faithful $G$-lattice $M$ and a non-split exact sequence of $G$-modules
    \begin{equation}\label{eq:seq-connecting-zero}
    1\longrightarrow k^\times \longrightarrow k[M]^\times \longrightarrow M \longrightarrow 0
    \end{equation}
    such that for all $n\geq 1$, the connecting map $\rH^n(G,M)\to \rH^{n+1}(G,k^\times)$ is zero.
\end{lemma}

By definition, $k[M]^\times$ is the group of units in the Laurent polynomial ring $k[M]$. Since $M$ is $\Z$-free, any element of 
$$
\on{Ext}^1_G(M,k^\times)=\rH^1(G,\on{Hom}(M,k^\times))
$$ 
is represented by a short exact sequence of the form \eqref{eq:seq-connecting-zero}. 

\begin{proof}
    Recall that  \[\rH^*(G,\F_2)=\F_2[x,y],\qquad |x|=|y|=1.\]
    Let $A,B\in \rH^2(G,\Z)$ be the images of $x,y$ under the Bockstein homomorphism $\rH^1(G,\F_2)\to \rH^2(G,\Z)$. The reduction modulo $2$ of $A$ is equal to $\on{Sq}^1(x)=x^2$, and the reduction modulo $2$ of $B$ is equal to $\on{Sq}^1(y)=y^2$. Since the abelian group $G$ has exponent equal to $2$, by the K\"unneth formula, the groups $\rH^i(G,\Z)$ are $2$-torsion, for all $i\geq 1$. The long exact sequence associated to the short exact sequence
    \[0\longrightarrow \Z\longrightarrow \Z \longrightarrow \Z/2\Z \longrightarrow 0\]
    shows that the reduction map $\rH^i(G,\Z)\to \rH^i(G,\F_2)$ is injective, for all $i\geq 1$. Since $\rH^*(G,\F_2)$ is a domain and the reduction mod $2$ of $A$ is equal to $x^2$, we deduce that 
    \begin{equation}\label{eq:key-kernel}
        \on{Ker}[\rH^*(G,\Z)\xlongrightarrow{A\cup(-)}\rH^{*+2}(G,\Z)]= 2\cdot \rH^0(G,\Z).
    \end{equation}
    
    Let $P_\bullet \to \mathbb{Z}$ be a projective resolution of the trivial $G$-module $\mathbb{Z}$, and $w \in \text{Ext}^2_G(\mathbb{Z}, \Omega^2 \mathbb{Z})$ the extension class of the exact sequence \[0 \longrightarrow \Omega^2 \mathbb{Z} \longrightarrow P_1 \longrightarrow P_0 \longrightarrow \mathbb{Z} \longrightarrow 0.\]
We denote by $\circ$ the composition (Yoneda product) of extensions of $G$-modules. For every $G$-module $L$ and every $n\geq 2$, composition with $w$ provides a surjective map 
\begin{equation}\label{eq:w}
(-)\circ w \colon \text{Ext}^{n-2}_G(\Omega^2\Z,L) \longrightarrow \text{Ext}^n_G(\Z,L)=\rH^n(G,L),
\end{equation}
which is an isomorphism if $n\geq 3$. In particular, there exist $a \in \text{Hom}_G(\Omega^2 \mathbb{Z}, \mathbb{Z})$ and $b \in \text{Ext}^2_G(\Omega^2 \mathbb{Z}, \mathbb{Z})$ such that $A=a\circ w$ and $B^2=b\circ w$. By replacing $\Omega^2 \mathbb{Z}$ with $\Omega^2 \mathbb{Z} \oplus \mathbb{Z}[G]$ and extending $a$ to map the generator of $\mathbb{Z}[G]$ to $1\in \Z$, if necessary, we may assume that $a$ is surjective. We define the $G$-lattice $M$ as the kernel of $a$, so that we have a short exact sequence of $G$-modules
\[0\longrightarrow M \xlongrightarrow{i} \Omega^2\Z \xlongrightarrow{a} \Z \longrightarrow 0.\]
Replacing $\Omega^2\Z$ by $\Omega^2\Z\oplus \Z[G]$, if needed, we may assume that $M$ is faithful. We obtain an exact sequence
\[\on{Ext}^2_G(\Z,\Z)\xlongrightarrow{(-)\circ a} \on{Ext}^2_G(\Omega^2\Z,\Z)\xlongrightarrow{(-)\circ i} \on{Ext}^2_G(M,\Z).\]
Put
$$
c_{\Z}\coloneqq b\circ i\in \on{Ext}^2_G(M,\Z).
$$
Observe that $c_{\Z}\neq 0$: indeed, if $b\circ i=0$, then $b=b'\circ a$ for some $b'\in \on{Ext}^2_G(\Z,\Z)$, and hence 
$$
B^2=b\circ w=b'\circ a \circ w=b'\circ A \quad \text{ in } \quad \on{Ext}^2_G(\Z,\Z),
$$
that is, $B^2$ is a multiple of $A$ in $\rH^*(G,\Z)$, a contradiction.

Since $k$ is algebraically closed of characteristic zero, we may fix a projection $k^\times \to \Q/\Z$ with uniquely divisible kernel $\mathbb K$. Pulling back the class of 
\[0\longrightarrow \Z \longrightarrow \Q \longrightarrow \Q/\Z \longrightarrow 0\]
along this projection gives an extension class $e\in \on{Ext}^1_G(k^\times,\Z)$. Because $\mathbb K$ is uniquely divisible, for every $G$-module $L$ and all $n\geq 1$ we have an isomorphism
\[e\circ(-)\colon\on{Ext}^n_G(L,k^\times)\xrightarrow{\sim}\on{Ext}^{n+1}_G(L,\Z).\]
Let $c\in \on{Ext}^1_G(M,k^\times)$ be the unique class such that $e\circ c=c_\Z$, and define the sequence \eqref{eq:seq-connecting-zero} as a representative of $c$. (This uniquely determines \eqref{eq:seq-connecting-zero}, up to isomorphism.) Since $c_{\Z}\neq 0$, we have $c\neq 0$, that is, \eqref{eq:seq-connecting-zero} is not split.

It remains to show that the connecting maps 
$$
c\circ(-)\colon\rH^n(G,M)\to \rH^{n+1}(G,k^\times)
$$ 
are zero, for all $n\geq 1$. 

Let $v \in \rH^n(G, M) =\on{Ext}^n_G(\mathbb{Z}, M)$ be any class. Then $c\circ v = 0$ if and only if  $e \circ (c \circ v) = 0$, that is, if and only if $c_{\Z}\circ v=0$.
We have 
$$ c_{\mathbb{Z}} \circ v = (b \circ i) \circ v = b \circ (i \circ v). $$
Consider the intermediate morphism $i \circ v \in \text{Ext}^n_G(\mathbb{Z}, \Omega^2 \mathbb{Z})$. When $n=1$, we have $$
\on{Ext}^1_G(\Z,\Omega^2\Z)=\hat{\rH}^{-1}(G,\Z)=0,
$$
where $\hat{\rH}$ denotes Tate cohomology. We may thus assume that $n\geq 2$. By \eqref{eq:w}, there exists a class $u \in \text{Ext}^{n-2}_G(\mathbb{Z}, \mathbb{Z})=\rH^{n-2}(G,\Z)$ such that
$$ i \circ v = w \circ u. $$
Apply $a$ to this equation. Since $a \circ i = 0$, we have
$$
    0 = (a\circ i)\circ v= a \circ (i \circ v) = a \circ (w \circ u) = (a \circ w) \circ u = A \circ u.
$$
In the cohomology ring $\rH^*(G, \mathbb{Z})$, the product $A \circ u$ is exactly the cup product $A \cup u$. We have thus proven that $A \cup u = 0$. Since $u$ has degree $n-2$, by \eqref{eq:key-kernel}, this implies that $u=0$, if $n\geq 3$, and $u=2u'$, for some $u'\in \Z=\rH^0(G,\Z)$, if $n=2$.

We return to the evaluation of the connecting map $c_{\mathbb{Z}} \circ v$:
\[c_{\mathbb{Z}} \circ v = b \circ (i \circ v) = b \circ (w \circ u) = (b \circ w) \circ u = B^2 \circ u.\]
It remains to prove that $B^2\cup u=0$ in $\rH^{n+2}(G,\Z)$. Recall that we have reduced to the case $n\geq 2$ and $A \cup u = 0$. If $n\geq 3$, then we showed that $u=0$. If $n=2$, then $u=2u'$, for some $u'\in \Z$, so that 
$$
B^2\cup u= B^2\cup (2u')=(2B^2)\cup u'=0.
$$
Thus $B^2\cup u=0$, as desired.
\end{proof}

\begin{lemma}\label{prop:toric-realization}
Let $k$ be a field, $G$ a finite group, $M$ a faithful $G$-lattice, and 
\[
1\longrightarrow k^\times \longrightarrow k[M]^\times \longrightarrow M \longrightarrow 0
\]
a short exact sequence of $G$-modules. Set $M^\vee\coloneqq \Hom_{\mathbb Z}(M,\mathbb Z)$. Then there exist:
\begin{itemize}
    \item a $G$-free $G$-stable finite subset $\mc{S}\subset M^\vee\setminus\{0\}$ of primitive elements which span $M^\vee_{\mathbb Q}$;
    \item a smooth rational $G$-variety $Y$ over $k$ containing $U\coloneqq \Spec(k[M])$ as a $G$-stable dense open subscheme, and
    \item for each $v\in \mc{S}$, a prime divisor $D_v\subset Y$,
\end{itemize}
such that the $G$-equivariant homomorphism
\[
\iota\colon M\longrightarrow \Z[\mc{S}],
\qquad
m\longmapsto \sum_{v\in \mc{S}}\langle m,v\rangle v,
\]
is injective and we have a commutative diagram with exact rows
\[
\begin{tikzcd}
    1 \arrow[r] & k^\times \arrow[r] \arrow[d,equal]  & k[M]^\times \arrow[r,"\iota"] \arrow[d,equal] & \Z[\mc{S}] \arrow[r] \arrow[d,"\wr"] & \on{Coker}(\iota) \arrow[d,"\wr"] \arrow[r] & 0 \\
    1 \arrow[r] & k^\times \arrow[r] & k[U]^\times \arrow[r,"\on{div}"] & \on{Div}_{Y\setminus U}(Y) \arrow[r] &  \on{Pic}(Y) \arrow[r] & 0,
\end{tikzcd}
\]
where the third vertical map is given by $v\mapsto D_v$. In particular, $k[Y]^\times =k^\times$ and the $G$-lattice $\on{Div}_{Y\setminus U}(Y)$ is $G$-free.
\end{lemma}

\begin{proof}
Let $N\coloneqq M^\vee$. Since $M$ is faithful, so is $N$. For each $1\neq g\in G$, the fixed subspace $N_\mathbb{Q}^g$ is a proper subspace of $N_\mathbb{Q}$. Hence the union $\cup_{1\neq g\in G} N_\Q^g$ is a finite union of proper subspaces of $N_\Q$. Therefore, we may choose primitive elements $v_1,\dots,v_r\in N$ outside this union such that they span $N_\mathbb{Q}$. Then each $v_i$ has trivial stabilizer, so each $G$-orbit $Gv_i$ is free. Then $\mc{S}\coloneqq \cup_{i=1}^r Gv_i$ is a $G$-free $G$-stable finite subset of $N\setminus\{0\}$ consisting of primitive elements, spanning $N_\mathbb{Q}$. For all $v\in \mc{S}$, let $\rho_v\coloneqq \mathbb{R}_{\geq 0}v$ be the ray spanned by $v$. Define a fan $\Sigma$ in $N_\mathbb{R}$ by $\Sigma\coloneqq \{0\}\cup \{\rho_v\mid v\in \mc{S}\}$. This is a fan since its only cones are $0$ and one-dimensional cones, and distinct rays intersect only in $0$.

Let $Y$ be the toric variety associated to $\Sigma$. For each $\sigma\in\Sigma$, define
\[
U_\sigma\coloneqq \Spec(k[\sigma^\vee\cap M]),
\quad
\sigma^\vee=\{m\in M_\mathbb{R}\mid \langle m,n\rangle \geq 0 \text{ for all } n\in \sigma\}.
\]
Then $Y$ is obtained by gluing the affine schemes $U_\sigma$ along their common open torus $U$. In particular, $Y$ is rational, and it is smooth because each cone is generated by a primitive vector (and hence in particular by a subset of a basis of $N_{\Q}$). The $G$-action on $M$ induces a $G$-action on $N$ preserving $\mc{S}$ and hence $\Sigma$, so $G$ acts on $Y$ extending its action on $U$.

Since $\mc{S}$ spans $N_\mathbb{Q}$, the map $\iota$ is injective.
 For each $v\in \mc{S}$, let $D_v$ be the corresponding torus-invariant prime divisor. Then
\[
\on{Div}_{Y\setminus U}(Y)=\bigoplus_{v\in \mc{S}}\mathbb{Z}D_v,
\]
and hence the map $\Z[\mc{S}]\to \on{Div}_{Y\setminus U}(Y)$ given by $v\mapsto D_v$ is an isomorphism. For every $m\in M$, the divisor of the character $\chi^m\in k[M]^\times$ is
\[
\on{div}(\chi^m)=\sum_{v\in S} \langle m,v\rangle D_v,
\]
where $\langle -,-\rangle\colon M\times N\to \Z$ is the tautological pairing.
This proves the commutativity of the middle square. It follows that the induced map $\on{Coker}(\iota)\to \on{Pic}(Y)$ is an isomorphism, as desired.
\end{proof}

\begin{proof}[Proof of \Cref{thm:connecting-maps-are-zero}]
    Let $M$ be a $G$-lattice as in \Cref{lem:all-connecting-maps-are-zero}. By \Cref{prop:toric-realization}, there exists a smooth rational $G$-variety $Y$ over $k$ with $k[Y]^\times=k^\times$, together with a dense open $G$-invariant subscheme $U=\on{Spec}(k[M])$ such that the $G$-lattice $\on{Div}_{Y\setminus U}(Y)$ is $G$-free and the sequence 
    \begin{equation}\label{eq:differentials-are-zero-x-u}
    1 \longrightarrow k^\times \longrightarrow k[U]^\times \xlongrightarrow{\on{div}} \on{Div}_{Y\setminus U}(Y) \longrightarrow \on{Pic}(Y) \longrightarrow 0
    \end{equation}
    is isomorphic to the composition of \eqref{eq:seq-connecting-zero} and the sequence
    \[0\longrightarrow M \longrightarrow \on{Div}_{Y\setminus U}(Y)  \longrightarrow \on{Pic}(Y)\longrightarrow 0.\]
    Now, \Cref{lem:all-connecting-maps-are-zero} implies that all double connecting maps are zero for \eqref{eq:differentials-are-zero-x-u}.    
    By \Cref{lem:finitepresent}, we know that \eqref{eq:differentials-are-zero-x-u} is equivalent to \eqref{elementary-obstruction-sequence}. In particular, all double connecting maps for \eqref{elementary-obstruction-sequence} are zero, that is, $\Am^n(Y\actsfromright G,\mathbb{G}_m)=0$ for all $n\geq 2$. 
    
    We show that $\beta(Y\actsfromright G)=0$.  Since $\beta(Y\actsfromright G)$ is the class of \eqref{elementary-obstruction-sequence}, we see that $\beta(Y\actsfromright G)=0$ if and only if \eqref{eq:differentials-are-zero-x-u} splits, that is, if and only if the class of \eqref{eq:seq-connecting-zero} in $\on{Ext}^1_G(M,k^\times)$ belongs to the kernel of $\on{Ext}^1_G(M,k^\times)\to \on{Ext}^2_G(\on{Pic}(Y),k^\times)$. We have an exact sequence
    \[\on{Ext}^1_G(\on{Div}_{Y\setminus U}(Y),k^\times) \longrightarrow \on{Ext}^1_G(M,k^\times)\longrightarrow \on{Ext}^2_G(\on{Pic}(Y),k^\times).\]
    Since $\on{Div}_{Y\setminus U}(Y)$ is $G$-free, we have $\on{Ext}^1_G(\on{Div}_{Y\setminus U}(Y),k^\times)=0$. It follows that $\beta(Y\actsfromright G)=0$ if and only if \eqref{eq:differentials-are-zero-x-u}. \Cref{lem:all-connecting-maps-are-zero} implies that $\beta(Y\actsfromright G)\neq 0$, as desired. By \Cref{prop:nonproper-toric-KT}, $Y$ is not $G$-unirational.

    Finally, let $X$ be a smooth proper $G$-equivariant compactification of $Y$. The open immersion $i\colon Y\hookrightarrow X$ induces a commutative diagram with exact rows
    \[
    \begin{tikzcd}
        1 \arrow[r] & k^\times\arrow[r] \arrow[d,equal] & k(X)^\times \arrow[r]\arrow[d,"\wr"',"i^*"] & \on{Div}(X) \arrow[r]\arrow[d,"i^*"] & \on{Pic}(X)\arrow[d,"i^*"] \arrow[r] & 0 \\
        1 \arrow[r] & k^\times \arrow[r] & k(Y)^\times \arrow[r] & \on{Div}(Y) \arrow[r] & \on{Pic}(Y) \arrow[r] & 0.
    \end{tikzcd}
    \]

    For all $n\geq 2$, we obtain commutative squares
    \[
    \begin{tikzcd}
        \rH^{n-2}(G,\on{Pic}(X)) \arrow[r,"\partial^n"] \arrow[d,"i^*"]  & \rH^n(G,k^\times) \arrow[d,equal] \\
        \rH^{n-2}(G,\on{Pic}(Y)) \arrow[r,"\partial^n"] & \rH^n(G,k^\times),
    \end{tikzcd}
    \]
    where we use the fact that $k[X]^\times$ and $k[\tilde{X}]^\times$ are both equal to $k^\times$. It follows that $$
    \Am^n(X\actsfromright G,\mathbb{G}_m)\subseteq \Am^n(Y\actsfromright G,\mathbb{G}_m).
    $$
    Since $\Am^n(Y\actsfromright G,\mathbb{G}_m)=0$, we deduce that $\Am^n(X\actsfromright G,\mathbb{G}_m)=0$, as desired. Since $Y$ is not $G$-unirational, $X$ is not $G$-unirational as well, and \Cref{prop:nonproper-toric-KT} implies that $\beta(X\actsfromright G)\neq 0$.
\end{proof}

\section{A non-\texorpdfstring{$G$}{G}-unirational Del Pezzo surface of degree 2}\label{sect:exam3}

Let $G$ be
the modular (maximal-cyclic) group of order $16$, 
given by 
\begin{equation}\label{eq:m16}
G\coloneqq \langle\sigma,\tau \mid \sigma^8=\tau^2=\tau\sigma\tau\sigma^3=e\rangle.
\end{equation}
It is the internal semidirect product $\ang{\sigma}\rtimes \ang{\tau}\simeq (\Z/8\Z)\rtimes (\Z/2\Z)$, where $\tau$ acts on $\sigma$ via $\tau\sigma\tau^{-1}=\sigma^5$. 

\begin{prop}\label{prop:del-pezzo}
Let $k$ be a field containing a root of unity $\zeta_8$ of order $8$, and of characteristic not two. Let 
$$
X\coloneqq \{w^2=x_1^4+x_2^4+x_3^4\}\subset \mathbb P(2,1,1,1),
$$
be the del Pezzo surface of degree $2$ over $k$, 
in the weighted projective space with coordinates $w$, $x_1$, $x_2$ and $x_3$, homogeneous of degree $2$, $1$, $1$, and $1$, respectively. 
The group $G$ of \eqref{eq:m16} acts on $X$ via 
$$
\sigma:(w,x_1,x_2,x_3)\mapsto (w,\zeta_4x_1,x_3,\zeta_4^3x_2),\quad \tau:(w,x_1,x_2,x_3)\mapsto (w,-x_1,-x_2,x_3).
$$
Then:
\begin{enumerate}
    \item every abelian subgroup of $G$ has fixed points on $X$,
    \item the $G$-action on $X$ is minimal, i.e., $\Pic(X)^G=\bZ$, and fixed point free,
    \item\label{prop:part_3}  $\beta(X\actsfromright G)\ne 0 \in\rH^2(G,T_\mathrm{NS}(k)),$ and
\item\label{prop:part_4} when $k$ is algebraically closed, $\Am^n(X\actsfromright G,\bG_m)=0$, for all $n\ge 2$.
\end{enumerate}
In particular, $X$ is not $G$-unirational. 
\end{prop}

\begin{rmk}\label{rmk:why-this-dp-2}
  By \cite[Section 3, Table 4]{TZ-uni}, when $\bar k=k$ and $\mathrm{char}(k)=0$, this is the unique action on a del Pezzo surface that satisfies properties (1) -- (4); indeed, by \cite[Theorem 1.4]{Duncan}, every del Pezzo surface of degree $\ge 3$ with a generically free action of a finite group $G$ satisfying (1) is unirational and  
del Pezzo surfaces of degree 1 have a fixed point, for every action, and thus $\beta(X\actsfromright G)=0$ in these cases. 
It suffices to consider del Pezzo surfaces of degree 2; an examination of all possible actions yields the result.  It was previously unknown whether or not this action on $X$ is $G$-unirational. 
\end{rmk}   

The rest of this section is devoted to a proof of \Cref{prop:del-pezzo}. Properties (1) and (2) can be checked immediately. Our proofs of parts (3) and (4) rely on {\tt Magma} and {\tt GAP} computations. 

\begin{proof}[Proof of Proposition~\ref{prop:del-pezzo}(\ref{prop:part_3})]
Let $M$ be a $\bZ[G]$-module. The  cohomology group $\rH^i(G,M)$, for $i=0,\ldots4$, can be computed as the cohomology of the complex 
\begin{equation}\label{eqn:resolutionod16}
 M \stackrel{d^0}{\longrightarrow}
  M^{\oplus2}\stackrel{d^1}{\longrightarrow}
 M^{\oplus2}\stackrel{d^2}{\longrightarrow}
 M^{\oplus2}\stackrel{d^3}{\longrightarrow}M^{\oplus3}\stackrel{d^4}{\longrightarrow}M^{\oplus4}\to\cdots
\end{equation}
where 
\begin{align*}
    d^0&=\begin{pmatrix}
  1-\sigma& 1-\tau
  \end{pmatrix},\\
  d^1&=\begin{pmatrix}
  0& 1+\sigma+\sigma^2+\sigma^7\tau\\
  1+\tau& \sigma^3-1
  \end{pmatrix},\\
  d^2&=\begin{pmatrix}
  \tau-1& 1-\sigma^3+\sigma^2-\sigma^7\\
 0& 1-\sigma^7\tau+\tau-\sigma^7
  \end{pmatrix},\\
  d^3&=\begin{pmatrix}
      1+\tau&-1-\sigma^2+\sigma^3+\sigma^7&1+\sigma^4-\sigma^5-\sigma^7\\
      0&\tau-1&1+\sigma+\sigma^2+\sigma^3+\sigma^4+\sigma^6+\sigma^7+\sigma^4\tau\sigma
  \end{pmatrix},\\
   d^4&=\begin{pmatrix}
      \tau-1&0&1+\sigma^2-\sigma^3-\sigma^7&1+\sigma^4-\sigma-\sigma^3\\
      0&\sigma^5-\sigma^6&1+\tau&1+\sigma^2+\sigma^3+\sigma^4+\sigma^6+\sigma^7\\
      0&\sigma-1&0&1-\tau
  \end{pmatrix}. 
\end{align*}
To compute $\beta(X\actsfromright G)$,  
consider the $G$-equivariant exact sequence 
\begin{align}\label{seq:dp2elem}
    1 \longrightarrow k^\times \longrightarrow k[U]^\times \longrightarrow \bigoplus_{i=1}^{56}\Z\cdot D_i\longrightarrow\Pic(X) \longrightarrow 0,
\end{align}
where $D_i$ are the 56 exceptional curves on $X$.
We break the tensor product of \eqref{seq:dp2elem} with $\Pic(X)^\vee$ into the short exact sequences 
\begin{align}\label{seq:break1}
    1\to T_{\mathrm{NS}}(k)\to k[U]^\times\otimes_{\Z}\Pic(X)^\vee\to R\otimes_{\Z}\Pic(X)^\vee\to 0,
    \end{align}
\begin{align}\label{seq:break2}
    0\to R\otimes_{\Z}\Pic(X)^\vee\to (\bigoplus_{i=1}^{56}\Z\cdot D_i)\otimes_{\Z}\Pic(X)^\vee\to \End(\Pic(X))^G
\to 0,
\end{align}
where $R$ is the $G$-module generated by relations among exceptional curves. 
By Proposition~\ref{prop:connect} and \Cref{lem:finitepresent}, $\beta(X\actsfromright G)$ can be computed as the image of $-\mathrm{Id}_{\Pic(X)}$ under the composition of connecting homomorphisms arising from~\eqref{seq:break2} and~\eqref{seq:break1}, respectively:
$$
\End(\Pic(X))^G\stackrel{\delta_1}{\longrightarrow}\rH^1(G, R\otimes\Pic(X)^\vee)\stackrel{\delta_2}{\longrightarrow} \rH^2(G, T_{\mathrm{NS}}(k)).
$$
 Concretely, the exceptional curves on $X$ are preimages of the bitangents to the ramification divisor on $\mathbb P^2$. They are given by 
\begin{align*}
    L_{i,j}^+&=\{x_1+\zeta_4^ix_2+\zeta_4^jx_3=w+\sqrt2((-1)^ix_2^2+
    \zeta_4^{i+j}x_2x_4+(-1)^jx_4^2)=0\},\\
      L_{i,j}^-&=\{x_1+\zeta_4^ix_2+\zeta_4^jx_3=w-\sqrt2((-1)^ix_2^2+
    \zeta_4^{i+j}x_2x_4+(-1)^jx_4^2)=0\},
\end{align*}
for $i,j\in\{1,2,3,4\}$ and 
\begin{align*}
    R_{1,i}^+&=\{x_1+\zeta_8^ix_2=w+x_3^2\},\quad
    R_{1,i}^-=\{x_1+\zeta_8^ix_2=w-x_3^2\},\\
     R_{2,i}^+&=\{x_2+\zeta_8^ix_3=w+x_1^2\},\quad
    R_{2,i}^-=\{x_2+\zeta_8^ix_3=w-x_1^2\},\\
     R_{3,i}^+&=\{x_3+\zeta_8^ix_1=w+x_2^2\},\quad
    R_{3,i}^-=\{x_3+\zeta_8^ix_1=w-x_2^2\}
\end{align*}
for $i\in\{1,3,5,7\}$. We pick the following lift of a basis of $\Pic(X)$:
$$
\eta_1=L_{4,1}^+,\quad 
\eta_2=L_{3,3}^-,\quad 
\eta_3=R_{2,1}^-,\quad
\eta_4=L_{2,1}^+,
$$
$$
\eta_5=R_{1,5}^+,\quad
\eta_6=L_{3,2}^-,\quad
\eta_7=L_{3,1}^+,\quad
\eta_8=L_{3,1}^-.
$$
Denote the corresponding dual basis of $\Pic(X)^\vee$ by $e_i$.  The image of $\mathrm{Id}_{\Pic(X)}$ under $\delta_1$ is represented by 
$$
\left((1-\sigma)\sum_{i=1}^8 \eta_i\otimes e_i,\quad (1-\tau)\sum_{i=1}^8 \eta_i\otimes e_i)\right)\in (R\otimes\Pic(X)^\vee)^2.
$$
The following computations rely on {\tt Magma}. The image under $\delta_2$, which is the class $-\beta(X\actsfromright G)$, is represented by 
\begin{multline*}
 \bigg(  (\zeta_8^3 - \zeta_8 + 1)\otimes (e_1+e_3)+(\zeta_8^3 - \zeta_8 - 1)\otimes e_4,\qquad \frac18(5\zeta_8^3 + 5\zeta_8^2 + 2\zeta_8 - 2)\otimes e_1+\\
+(-4\zeta_8^3 + 12\zeta_8^2 - 12\zeta_8 + 4)\otimes e_2+
(-16\zeta_8^3 + 16\zeta_8^2 - 8\zeta_8 - 8)\otimes e_3+\\+
(-96\zeta_8^3 + 96\zeta_8^2 - 40\zeta_8 - 40)\otimes e_4+
\frac18(3\zeta_8^3 +\zeta_8^2 -\zeta_8 - 3)\otimes e_5+\\
+(2\zeta_8^3 - 6\zeta_8^2 + 6\zeta_8 - 2)\otimes e_6+
(8\zeta_8^3 - 8\zeta_8^2 + 4\zeta_8 + 4)\otimes e_7+\\
+\frac 18(-3\zeta_8^3 -\zeta_8^2 +\zeta_8 + 3)\otimes e_8\bigg)\in (k^\times\otimes\Pic(X)^\vee)^2=(T_{\mathrm{NS}}(k))^2.
\end{multline*}
Using the complex \eqref{eqn:resolutionod16}, one can verify that this indeed represents a 2-cocycle which is a nontrivial 2-torsion class in $\rH^2(G,T_{\mathrm{NS}}(k))$.
\end{proof}
\begin{proof}[Proof of Proposition~\ref{prop:del-pezzo}(\ref{prop:part_4})]
Let 
    $$
   \mathrm{B}^n(G,\bZ)\coloneqq \bigcap_{A\subset G}\mathrm{Ker}\left(\rH^n(G,\bZ)\xrightarrow{\makebox[0.7cm]{$\scriptstyle\mathrm{Res}_{A}$}}\rH^n(A,\bZ)\right),
    $$
    where $G$ acts trivially on $\bZ$, and $A$ runs over abelian subgroups of $G$. Since $k$ is algebraically closed, by Property (1) in~\Cref{prop:del-pezzo} and Property (2) in Theorem~\ref{thm:summary}, we know that 
    $$
    \mathrm{Am}^n(X\actsfromright G,\mathbb{G}_m)\subseteq \mathrm{B}^{n+1}(G,\bZ), \quad \text{ for all $n\geq 2$}.
    $$
 The restriction homomorphism respects the cup product in $\rH^*(G,\bZ)$, and thus 
$$  \mathrm B(G,\bZ)\coloneqq \bigoplus_{n=0}^\infty\mathrm B^n(G,\bZ)=\bigcap_{A\in \mathcal A}\mathrm{Ker}\left(\rH^*(G,\bZ)\xrightarrow{\makebox[0.7cm]{$\scriptstyle\mathrm{Res}_{A}$}}\rH^*(A,\bZ)\right),
    $$
    where $\mathcal A:=\{G_1,G_2,G_3\}$ is the set of  maximal abelian subgroups of $G$, i.e., 
    $$
G_1=\langle \sigma\rangle\simeq \bZ/8,\quad G_2=\langle\tau\sigma\rangle\simeq \bZ/8,\quad G_3=\langle \tau,\sigma^2\rangle\simeq \bZ/2\times\bZ/4.
$$
   It suffices to show that $\mathrm B(G,\bZ)=0$.  By \cite[Corollary 6]{Modularring}, we know that 
$$
\rH^*(G,\bZ)=\bZ[z_1,z_2,z,q]/I, \quad |z_1|=|z_2|=2, \quad|z|=4, \quad |q|=5, 
   $$
   where 
   $$
   I=(2z_1,4z_2,8z,2q,z_1z_2,z_2^2-4z,z_2q,q^2).
   $$
We also have 
$$
\rH^*(G_1,\bZ)\simeq \bZ[t_1]/(8t_1),\quad \rH^*(G_2,\bZ)\simeq \bZ[t_2]/(8t_2),\quad |t_1|=|t_2|=2, 
$$
$$
\rH^*(G_3,\bZ)\simeq \bZ[y_1,y_2,y_3]/(2y_1,4y_2,2y_3,y_3^2),\quad |y_1|=|y_2|=2,\quad |y_3|=3. 
$$
For a finite group $H$, we identify 
$\rH^2(H,\bZ)$ with $\rH^1(H,\mathbb Q/\bZ)=\Hom(H,\mathbb Q/\mathbb Z)$, the character group of $H$. We may choose the degree 2 generators such that 
$$
  t_1(\sigma)=\zeta_8,\quad t_2(\tau\sigma)=\zeta_8,  
  $$
  $$
  y_1(\tau)=-1,\quad y_1(\sigma^2)=1,\quad  y_2(\tau)=1,\quad y_2(\sigma^2)=\zeta_4,
  $$
  $$
 z_1(\tau)=-1,\quad z_1(\sigma)=1,\quad z_2(\tau)=1,\quad z_2(\sigma)=\zeta_4.
  $$
Computing restrictions of characters, we obtain that
\begin{align}\label{eqn:res1}
       \mathrm{Res}_{G_1}(z_1)=0,\quad \mathrm{Res}_{G_2}(z_1)=4t_2,\quad \mathrm{Res}_{G_3}(z_1)=y_1,\\\label{eqn:res2}
      \mathrm{Res}_{G_1}(z_2)=2t_1,\quad \mathrm{Res}_{G_2}(z_2)=2t_2,\quad \mathrm{Res}_{G_3}(z_2)=2y_2.
\end{align}
To compute restrictions of $z$ to $G_i$, we note that the conjugation action of $G/G_3$ on $\rH^*(G_3,\bZ)$ is given by 
$$
y_1\mapsto y_1,\quad y_2\mapsto y_1+y_2,\quad y_3\mapsto y_3. 
$$
Since $\mathrm{Cores}_{G}(\mathrm{Res}_{G_3}(v))=2z$ has order $4$, we know that $\mathrm{Res}_{G_3}(z)$ is an element in $\rH^{4}(G,\bZ)^{G/G_3}$ of order 4. Thus, up to replacing $z$ by a suitable element in $\{-z,z+z_1^2,-z+z_1^2\}$,  we obtain that
\begin{align}\label{eqn:res3}
    \mathrm{Res}_{G_3}(z)=y_1^2+y_2^2+y_1y_2.
\end{align}
Let $G_{13}:=G_1\cap G_3\simeq \bZ/4$. We have that
$\rH^*(G_{13},\bZ)=\bZ[z_{13}]/(4z_{13})$ such that $ |z_{13}|=2
$
and $z_{13}(\sigma^2)=\zeta_4$. Observe that 
$$
\mathrm{Res}_{G_{13}}(y_1)=0,\quad  \mathrm{Res}_{G_{13}}(y_2)=z_{13},\quad \mathrm{Res}_{G_{13}}(t_1)=z_{13}.
$$
Since
$$
\mathrm{Res}_{G_{13}}(\mathrm{Res}_{G_1}(z))=\mathrm{Res}_{G_{13}}(\mathrm{Res}_{G_3}(z))=(\mathrm{Res}_{G_{13}}(y_2))^2=z_{13}^2,
$$
it follows that
\begin{align}\label{eqn:res4}
\mathrm{Res}_{G_1}(z)\in \{t_1^2,5t_1^2\}.
\end{align}
We also have that $G_2\cap G_3=G_{13}$, so that $\mathrm{Res}_{G_{13}}(\mathrm{Res}_{G_{2}}(z))=z_{13}^2$, and
\begin{align}\label{eqn:res5}
\mathrm{Res}_{G_2}(z)\in \{t_2^2,5t_2^2\}.
\end{align}
Using {\tt GAP},\footnote{We use the {\tt HAP} package, with the code: {\tt G:=SmallGroup(16,6);Bogomology(G,4);}} we verify that 
$$
\mathrm{Ker}\left(\rH^5(G,\bZ)\xrightarrow{\makebox[0.8cm]{$\scriptstyle\mathrm{Res}_{G_3}$}}\rH^5(G_3,\bZ)\right)=0.
$$
It follows from $\rH^5(G_3,\bZ)^{G/G_3}=\langle y_1y_3\rangle\simeq \bZ/2$ that
\begin{align}\label{eqn:res6}
\mathrm{Res}_{G_1}(q)=0,\quad \mathrm{Res}_{G_2}(q)=0,\quad \mathrm{Res}_{G_3}(q)=y_1y_3. 
\end{align}
 Putting together~\eqref{eqn:res1} --~\eqref{eqn:res6}, one can show that $\mathrm B(G,\bZ)=0$ by computing kernels of the restriction homomorphisms using elimination ideals (for all 4 choices coming from~\eqref{eqn:res4} and~\eqref{eqn:res5}).
\end{proof}

\begin{rmk}
In this case, one can also compute using {\tt GAP} that
\begin{align}\label{eqn:bogomolov}
    \bigcap_{A\in \mathcal A}\mathrm{Ker}\left(\rH^2(G,T_{\mathrm{NS}}(k))\xrightarrow{\makebox[0.8cm]{$\scriptstyle\mathrm{Res}_{A}$}}\rH^2(A,T_{\mathrm{NS}}(k))\right) \simeq \bZ/2\bZ,
\end{align}
and thus~\eqref{eqn:bogomolov} is generated by $\beta(X\actsfromright G)$.
\end{rmk}
\section{Higher Amitsur groups in the arithmetic setting}
\label{sect:geom-field}

In this section, we discuss arithmetic variants of Theorems \ref{thm:intro-invariant} and \ref{thm:main2}, in the context of the descent formalism of Colliot-Th\'el\`ene and Sansuc \cite{CTSansucDuke}. To the best of our knowledge, these variants have not been considered for $n\geq 3$ in the literature.

We begin by recalling the setup of \cite{CTSansucDuke}. Let $k$ be a field, $\bar{k}$ a separable closure of $k$, and $\mg_k\coloneqq \mathrm{Gal}(\bar{k}/k)$ the absolute Galois group of $k$. For every discrete $\mg_k$-module $M$ and every $i\geq 0$, we put 
$$
\rH^i(k,M)\coloneqq \rH^i(\mg_k,M).
$$
For every $k$-scheme $Y$, we set $\bar{Y}\coloneqq Y\times_k\bar{k}$. Let $T$ be a $k$-torus (not necessarily $k$-split), $\mathfrak{X}_*(\bar{T})$ the cocharacter $\mg_k$-lattice of $\bar{T}$, and $X$ a smooth geometrically integral $k$-variety. We have an exact sequence of discrete $\mg_k$-modules
\begin{equation}
\label{eqn:2}
1\longrightarrow \bar{k}[X]^\times \longrightarrow \bar{k}(X)^\times \xlongrightarrow{\mathrm{div}}  \mathrm{Div}(\bar{X}) \longrightarrow \Pic(\bar{X})\longrightarrow 0.
\end{equation}
Tensoring \eqref{eqn:2} with the $\Z$-free $\mg_k$-module $\mathfrak{X}_*(\bar{T})$ yields a sequence
\begin{equation}
\label{eqn:3}
1\longrightarrow T(\bar{k}[X]) \longrightarrow T(\bar{k}(X)) \xlongrightarrow{\mathrm{div}} \mathrm{Div}(\bar{X})\otimes\mathfrak{X}_*(\bar{T}) \longrightarrow \Pic(\bar{X})\otimes\mathfrak{X}_*(\bar{T})\longrightarrow 0,
\end{equation}
with extension class
\[\alpha(X,T)\in \on{Ext}^2_{\mg_k}(\on{Pic}(\bar{X})\otimes\mathfrak{X}_*(\bar{T}),T(\bar{k}[X])),\]
and hence, for all $n\geq 2$, connecting maps
\begin{equation}\label{eqn:4}
    \partial^n\colon \rH^{n-2}(k,\on{Pic}(\bar{X})\otimes\mathfrak{X}_*(\bar{T}))\longrightarrow \rH^n(k,T(\Bar{k}[X])),\qquad c\mapsto c\cup \alpha(X,T).
\end{equation}
The Hochschild--Serre spectral sequence, that is, the Grothendieck spectral sequence for the composition of the global sections functor for abelian \'etale sheaves over $\Bar{X}$ with the functor of $\mg_k$-invariants, has the form
\begin{equation}\label{eq:galois-hochschild-serre}
\rE^{p,q}_2\coloneqq \rH^p(k,\rH^q(\bar{X},\bar{T}))\Longrightarrow \rH^{p+q}(X,T).
\end{equation}
For all $p\geq 0$, we obtain differentials
\begin{equation}\label{eqn:5}
    \mathrm{d}^{p,1}_2\colon \rH^p(k,\rH^1(\bar{X},\bar{T}))\longrightarrow \rH^{p+2}(k,T(\Bar{k}[X])).
\end{equation}
We also have the Grothendieck spectral sequence for the composition of the global sections functor for $\mg_k$-equivariant discrete abelian Zariski sheaves on $\Bar{X}$ and the functor of $\mg_k$-invariants: 
\begin{equation}\label{eqn:zariski-hochschild-serre}
\rE^{p,q}_2\coloneqq \rH^p(k,\rH^q_{\on{Zar}}(\bar{X},\bar{T}))\Longrightarrow \rH^{p+q}_{\mg_k,\on{Zar}}(\Bar{X},\Bar{T}),
\end{equation}
where $\rH^*_{\mg_k,\on{Zar}}$ denotes $\mg_k$-equivariant Zariski cohomology on $\Bar{X}$, that is, the derived functor of the functor sending a $\mg_k$-equivariant Zariski abelian sheaf $F$ on $\Bar{X}$ to $\Gamma(\Bar{X},\Bar{F})^{\mg_k}$. For all $p\geq 0$, we obtain differentials
\begin{equation}\label{eqn:zariski-hochschild-serre-differentials}
    (\mathrm{d}_{\mathrm{Zar}})^{p,1}_2\colon \rH^p(k,\rH^1_{\mathrm{Zar}}(\bar{X},\bar{T}))\longrightarrow \rH^{p+2}(k,T(\Bar{k}[X])).
\end{equation}
We have isomorphisms of discrete $\mg_k$-modules
\begin{equation}\label{eqn:6}
    \on{Pic}(\bar{X})\otimes_\Z\mathfrak{X}_*(\bar{T})\xlongrightarrow{\sim} \rH^1_{\on{Zar}}(\bar{X},\bar{T})\xlongrightarrow{\sim} \rH^1(\bar{X},\bar{T}),
\end{equation}
where the first map is given by $L\otimes\chi\mapsto \chi_*(L)$ and the second map is the change-of-site map, which is an isomorphism by Grothendieck's generalization of Hilbert's theorem 90.

\begin{thm}\label{thm:galois-delta=d}
		Let $k$ be a field, $T$ a $k$-torus, and $X$ a smooth geometrically integral $k$-variety. For all $n\geq 2$, we have a commutative square
    \[
    \begin{tikzcd}
        \rH^{n-2}(k,\on{Pic}(\bar{X})\otimes_\Z\mathfrak{X}_*(\bar{T})) \arrow[dr,"\partial^n"] \arrow[r,"\sim"] & \rH^{n-2}(k,\rH^1_{\on{Zar}}(\bar{X},\bar{T})) \arrow[d,"(\mathrm{d}_{\on{Zar}})_2^{n-2,1}"] \arrow[r,"\sim"] & \rH^{n-2}(k,\rH^1(\bar{X},\bar{T})) \arrow[dl,"\mathrm{d}_2^{n-2,1}"] \\
         & \rH^n(k,T(\bar{k}[X]))
    \end{tikzcd}
    \]
    where the top row is induced by \eqref{eqn:6} and the maps $\partial^n$, $(\mathrm{d}_{\on{Zar}})_2^{n-2,1}$ and $\mathrm{d}^{n-2,1}_2$ are the maps of \eqref{eqn:4}, \eqref{eqn:zariski-hochschild-serre-differentials} and \eqref{eqn:5}, respectively.
	\end{thm}
When $n=2$, this is \cite[Proposition 1.5.2 (iv)]{CTSansucDuke}. The general case (for $T=\mathbb{G}_m$) follows from \cite[Proposition 1.1 and (1.6)]{skoro}.

In view of \Cref{thm:galois-delta=d}, for all $n\geq 2$, we may define the \emph{$n$th Amitsur group} of the smooth geometrically integral $k$-variety $X$ with coefficients in $T$ as
\[\Am^n(X,T)\coloneqq \on{Im}(\partial^n)=\on{Im}((\mathrm{d}_{\on{Zar}})^{n-2,1}_2)=\on{Im}(\mathrm{d}_2^{n-2,1}).\]

\begin{proof}
    \'Etale descent gives an equivalence between the category of abelian \'etale sheaves on $X$ and the category of $\mg_k$-equivariant abelian \'etale sheaves on $\Bar{X}$. The change-of-site morphism $X_{\text{\'et}}\to X_{\text{Zar}}$ gives a morphism of spectral sequences from \eqref{eqn:zariski-hochschild-serre} to \eqref{eq:galois-hochschild-serre}, and hence proves the commutativity of the triangle on the right. In order to prove the commutativity of the left triangle, one replaces \eqref{eq:flasque-resolution-gm} by the exact sequence of $\mg_k$-equivariant abelian Zariski sheaves on $\Bar{X}$
    \[1\longrightarrow\mathbb{G}_{m,\,\Bar{X}}\longrightarrow\Bar{j}_*\mathbb{G}_{m,\,\Bar{k}(X)}\longrightarrow \mathrm{Div}_{\Bar{X}}\longrightarrow 0,\]
    where $\Bar{j}\colon \on{Spec}(\Bar{k}(X))\hookrightarrow \Bar{X}$ is the inclusion of the generic point, $\mathrm{Div}_{\Bar{X}}$ is the direct sum of the skyscraper sheaves $(i_x)_*\Z$, where $x$ ranges over all points of codimension $1$ in $\Bar{X}$, and $i_x\coloneqq \on{Spec}(k(x))\hookrightarrow \Bar{X}$ is the inclusion morphism. One then argues as in the proof of \Cref{thm:delta=d}.
\end{proof}

\begin{thm}\label{thm:galois-1}
    Let $k$ be a field of characteristic zero, $T$ a $k$-torus, and $X$ and $Y$ stably birationally equivalent smooth projective irreducible $k$-varieties. Then, for all $n\geq 2$, we have
    \[\Am^n(X,T)=\Am^n(Y,T).\]
\end{thm}

\begin{proof}
    It suffices to follow the proofs of \Cref{lem:pullback-alpha}, \Cref{lem:pullback-alpha-epsilon} and \Cref{prop:invariance-blowup-product}, replacing $k$ by $\Bar{k}$ and $G$ by $\mg_k$.
\end{proof}

Suppose now that the abelian group $\on{Pic}(\Bar{X})$ is free and finitely generated, and let $T_{\mathrm{NS}}$ be the N\'eron-Severi torus of $X$, that is, the $k$-torus whose character lattice is isomorphic to $\on{Pic}(\Bar{X})$ as a $\mg_k$-module. It follows that $\mathfrak X_*(\Bar{T}_{\mathrm{NS}})$ is isomorphic to $\on{Pic}(\Bar{X})^\vee$. 
    By \Cref{thm:galois-delta=d} for $n=2$ and $T=T_{\mathrm{NS}}$, we have a commutative triangle
    \[
    \begin{tikzcd}
        \on{End}(\on{Pic}(\Bar{X}))^{\mg_k} \arrow[d,"\partial^2"] \arrow[r,"\sim"]  & \rH^1(\Bar{X},\Bar{T}_{\mathrm{NS}})^{\mg_k} \arrow[dl,"\mathrm{d}_2^{0,1}"] \\
        \rH^2(k,T_{\mathrm{NS}}(\Bar{k}[X])) 
    \end{tikzcd}
    \]
Recall that the \emph{universal torsor obstruction} for the $k$-variety $X$ is defined as 
\begin{equation}\label{eqn:beta-ar}
\beta(X)\coloneqq\partial^2(\mathrm{Id}_{\on{Pic}(\Bar{X})})\in \rH^2(k,T_{\mathrm{NS}}(\Bar{k}[X])).
\end{equation}

\begin{thm}
\label{thm:galois-2}
Let $k$ be a field, $X$ a smooth projective geometrically connected $k$-variety such that $\on{Pic}(\Bar{X})$ is $\Z$-free and finitely generated. If $\beta(X)=0$, then $\Am^n(X,T)=0$, for all $n\ge 2$ and all $k$-tori $T$.     
\end{thm}

\begin{proof}
    We have an isomorphism
\begin{equation*}
    \rH^2(k,T_{\mathrm{NS}}(\Bar{k}[X]))\xlongrightarrow{\sim} \on{Ext}^2_{\mg_k}(\on{Pic}(\Bar{X}),\Bar{k}[X]^\times)
\end{equation*}
    which sends $\beta(X)$ to $-\alpha(X,T_{\mathrm{NS}})$. It now suffices to follow the proof of \Cref{thm:main2}, replacing \Cref{thm:delta=d} by \Cref{thm:galois-delta=d}.
\end{proof}

\begin{rmk}\label{rmk:difference}
    A fundamental difference between the $G$-equivariant setting and the arithmetic setting is that Hilbert's theorem 90 does not hold \enquote{over $\mathrm{B}G$}: for a finite group $G$ acting trivially on $k^\times$, the cohomology group $\rH^1(G,k^\times)$ is often nontrivial. 
    
    This leads to the following important difference between the descent formalism of \cite{CTSansucDuke} and the theory presented in this paper. By \cite[Proposition 2.2.4]{CTSansucDuke}, for a field $k$ and a geometrically integral smooth projective $k$-variety $X$, the vanishing of $\beta(X)$ is equivalent to the vanishing of the {\em elementary obstruction}, which by definition is the class of the short exact sequence
\[
1\longrightarrow \bar{k}^\times \longrightarrow \bar{k}(X)^\times \longrightarrow \bar{k}(X)^\times/\bar{k}^\times \longrightarrow 1
\]
in $\mathrm{Ext}^1_{\mg_k}(\bar{k}(X)^\times/\bar{k}^\times, \bar{k}^\times)$. The proof of this equivalence uses Hilbert's theorem 90. In contrast, an example of a finite group $G$ and an irreducible smooth projective $G$-variety $X$ over an algebraically closed field $k$ such that $\beta(X\actsfromright G)=0$ (that is, the extension \eqref{elementary-obstruction-sequence} is trivial) while the exact sequence of $G$-modules
\[
1\longrightarrow k^\times \longrightarrow k(X)^\times \longrightarrow k(X)^\times/k^\times \longrightarrow 1
\]
does not admit a $G$-equivariant splitting is given in \cite[Remark 5.3]{KT-uni}. This difference also intervenes in \Cref{sect:exam2}: it is the reason why we need to construct $Y$ such that $\on{Div}_{Y\setminus U}(Y)$ is $G$-free in \Cref{prop:toric-realization}.
\end{rmk}

\begin{rmk}\label{rmk-bridge}
    A bridge between birational geometry over non-closed fields and equivariant birational geometry was established by Duncan--Reichstein \cite{DR}. Let $k$ be a field, $G$ a finite group, and $X$ an integral $G$-variety over $k$. Let $V$ be a finite-dimensional faithful $G$-representation over $k$, and let $V^\circ\subset V$ be the $G$-stable dense open subscheme of $V$ over which $G$ acts freely. A strengthening of \cite[Theorem 1.1]{DR} in \cite[Proposition 2.1]{HT-pfister} implies that the $G$-action on $X$ is stably linearizable if and only if a twist of $X$ via the generic fiber of the $G$-torsor $V^\circ \to V^\circ/G$ over the field $K\coloneqq k(V/G)$ is stably rational over $K$. For convenience, we include the argument. Let $P\to \mathrm{Spec}(K)$ be the generic fiber of the $G$-torsor $V^\circ \to V^\circ/G$, and let $\prescript{P}{}{X}$ be the twist of $X$ by $P$. If $X$ is stably $G$-linearizable, then by twisting we immediately deduce that $\prescript{P}{}{X}$ is $K$-stably rational. Conversely, assume given a birational equivalence \[\A^m_K\stackrel{\sim}{\dashrightarrow} (\prescript{P}{}{X})\times_K \A^n_K=\prescript{P}{}{(X\times_k \A^n_k)}\] 
    for some $m,n\geq 0$. By \cite[Lemma 5.1]{DR}, there exists a $G$-equivariant birational equivalence $V\times_k\A^m_k\stackrel{\sim}{\dashrightarrow} V\times_kX\times_k\A^n_k$, where $G$ acts trivially on $\A^m_k$ and $\A^n_k$. It follows that $X$ is stably linearizable.

    Suppose that $k$ is algebraically closed and the integral $G$-variety $X$ is smooth over $k$. Put $Y\coloneqq \prescript{P}{}{X}$. Let $S$ be a $G$-torus and $T$ the $K$-torus given by $T=\prescript{P}{}{S}$. By Galois theory, we have a continuous surjective homomorphism 
    $$
    \rho\colon \mg_K\longrightarrow \on{Gal}(k(V)/K)=G.
    $$
    Pullback along $\rho$ gives a homomorphism
    \[\rho^*\colon \on{Ext}^2_G(\on{Pic}(X),S(k[X]))\longrightarrow \on{Ext}^2_{\mg_K}(\on{Pic}(\Bar{Y}),T(\Bar{K}[Y]))\]
    which maps $\alpha(X\actsfromright G,S)$ to $\alpha(Y,T)$.
    For all $n\geq 2$, this gives commutative squares
    \[
    \begin{tikzcd}
        \rH^{n-2}(G,\on{Pic}(X)\otimes\mathfrak{X}_*(S)) \arrow[r,"\rho^*"] \arrow[d,"\partial^n"] & \rH^{n-2}(K,\on{Pic}(\Bar{Y})\otimes\mathfrak{X}_*(\Bar{T})) \arrow[d,"\partial^n"] \\
        \rH^n(G,S(k[X])) \arrow[r,"\rho^*"] & \rH^n(K,T(\Bar{K}[Y])),
    \end{tikzcd}
    \]
    where the vertical arrows on the left (resp. right) are the double connecting homomorphisms \eqref{eq:partial-x-s} (resp. \eqref{eqn:4}). We obtain maps
    \begin{equation}\label{eq:bridge-amitsur}
        \rho^*\colon\Am^n(X\actsfromright G,S)\longrightarrow\Am^n(Y,T),
    \end{equation}
    which are not injective, in general. Indeed, $K$ has finite cohomological dimension $d=\on{dim}(V)$, and hence $\Am^n(Y,T)=0$, for all $n>d$, while it is possible for $\Am^n(X\actsfromright G,S)$ to be non-zero for all $n\geq 2$. This is an instance of the phenomenon of {\em negligible cohomology}, see, e.g.,  \cite{merkur} for a systematic study of this notion for 2-groups, and \cite{MerkScaviaJAMS} for the computation of negligible cohomology in degree $2$ with coefficients in an arbitrary finite $G$-module.
    
    Here is an explicit example. Let $k$ be an algebraically closed field of characteristic not equal to two. Consider the faithful action of the Klein $4$-group $G$ on $X\coloneqq \P^1_k$, as in \Cref{prop:p^1-amitsur}. Then $\Am^n(X\actsfromright G,\bG_m)\neq 0$, for all $n\geq 2$. On the other hand, letting $V$ be a faithful $2$-dimensional $G$-representation over $k$, the cohomological dimension of $K=k(V)^G$ is $2$, and therefore $\Am^n(Y,\bG_m)=0$, for all $n\geq 3$. Since the vanishing of the $n$th Amitsur group $\Am^n(-,\mathbb{G}_m)$ is a stably birational invariant of smooth projective $K$-varieties, this implies that $\Am^n(Y,\bG_m)=0$, for all $n\geq 3$ and for any choice of a finite-dimensional faithful $G$-representation $V$ over $k$.
\end{rmk}

\section*{Acknowledgments}
We are grateful to J.-L. Colliot-Th\'el\`ene for drawing our attention to \cite[Proposition~1.1]{skoro}, after receiving our manuscript.

	\bibliography{amitsur}
	\bibliographystyle{alpha}
	
\end{document}